\documentclass[a4paper]{amsart}
\usepackage[utf8]{inputenc}
\usepackage{amssymb}
\DeclareMathAlphabet\mathbfcal{OMS}{cmsy}{b}{n}
\usepackage[foot]{amsaddr}
\usepackage{stmaryrd}
\usepackage[new]{old-arrows}
\usepackage{enumitem}
\usepackage{tikz-cd}
\usepackage{url}
\usepackage[doi=false, backend=biber, style=alphabetic, sorting=nty, maxbibnames=5]{biblatex}
\DeclareFieldFormat[article]{title}{#1}
\DeclareFieldFormat{postnote}{#1}
\DeclareFieldFormat{multipostnote}{#1}
\usepackage[colorlinks, pdfpagelabels, urlcolor = blue, pdfstartview = FitH, bookmarksopen = true, bookmarksnumbered = true, linkcolor = black, plainpages = false, hypertexnames = false, citecolor = blue]{hyperref}
\usepackage{cleveref}
\usepackage{orcidlink}

\makeatletter
\newcommand*{\longdashrightarrow}{%
    \mathrel{\dabar@\dabar@\dabar@\mathchar"0\hexnumber@\symAMSa 4B}}
\makeatother

\newtheorem{theorem}{Theorem}[section]
\newtheorem{corollary}[theorem]{Corollary}
\newtheorem{lemma}[theorem]{Lemma}
\newtheorem{proposition}[theorem]{Proposition}

\newtheorem{claim}{Claim}

\theoremstyle{definition}
\newtheorem{definition}[theorem]{Definition}
\newtheorem{construction}[theorem]{Construction}

\newtheorem{step}{Step}

\theoremstyle{remark}
\newtheorem{remark}[theorem]{Remark}

\newcommand{\Z}{\mathbb{Z}}
\newcommand{\C}{\mathbb{C}}
\newcommand{\Q}{\mathbb{Q}}
\newcommand{\OO}{\mathcal{O}}
\newcommand{\pic}{\mathrm{Pic}}
\newcommand{\im}{\mathrm{im}}

\newcommand{\id}{\mathrm{id}}

\title{Holomorphic 1-forms without zeros on K\"ahler manifolds of Kodaira codimension one}
\author{Simon Pietig}
\address{Leibniz Universit\"at Hannover, Institut f\"ur Algebraische Geometrie, Welfengarten 1, 30167 Hannover}
\address{\href{mailto:pietig@math.uni-hannover.de}{pietig@math.uni-hannover.de}}
\address{\href{https://orcid.org/0009-0001-1800-6799}{Orcid iD}}

\begin{document}
\begin{abstract}
    We give a bimeromorphic classification of compact K\"ahler manifolds of Kodaira codimension one that admit a holomorphic 1-form without zeros.
\end{abstract}
\maketitle
\setcounter{tocdepth}{1}
\tableofcontents
\section{Introduction}
Let $X$ be a smooth projective variety over $\C$ with a nowhere vanishing holomorphic 1-form $\omega\in H^0(X,\Omega^1_X)$. In \cite{Popa_Schnell_Kodaira_dimension_zeros_1-forms}, Popa and Schnell showed that the Kodaira dimension of $X$ is bounded by $\kappa(X)\leq\dim X-1$. To prove this result, Popa-Schnell applied the theory of mixed Hodge modules. A fortiori, Hao gave the following geometric explanation of this behavior, for $\kappa(X)=\dim X-1$.
\begin{theorem}\label{Hao_Kodaira_codim_1}\cite[][Thm. 1.3]{Hao_Nowhere_Vanishing_Holomorphic_One-Forms_on_Varieties_of_Kodaira_Codimension_One}
    Let $X$ be a smooth projective variety of Kodaira dimension $\kappa(X)=\dim X-1$. If $X$ has a nowhere vanishing holomorphic 1-form, then for any minimal model $X^{\mathrm{min}}$ of $X$ there exists a finite quasi-\'etale covering $X'\rightarrow X^{\mathrm{min}}$ such that any $\Q$-factorialization $X''$ of $X'$ is a product $B\times E$, where $B$ is a minimal model of general type and $E$ is an elliptic curve.
\end{theorem}
Recent work of Church \cite{church_gmm} and of Hao, Wang and Zhang \cite{Hao_Wang_Zhang_gmm} generalized Hao's result to a birational classification of smooth projective varieties with $g$ everywhere linearly independent holomorphic 1-forms assuming the abundance conjecture without any restrictions of the Kodaira dimension of $X$.\par
The main result of our paper is the following generalization of \Cref{Hao_Kodaira_codim_1} to the K\"ahler case. For a more precise version, see \Cref{precise_main_result} down below.
\begin{theorem}\label{main_result}
    Let $X$ be a compact K\"ahler manifold of Kodaira dimension $\kappa(X)=\dim X-1$. If $X$ has a nowhere vanishing holomorphic 1-form, then there exists a finite \'etale cover $X'\rightarrow X$ such that $X'$ is bimeromorphic to an elliptic fiber bundle $Z\rightarrow B$ with trivial monodromy over a smooth projective base $B$ of general type.
\end{theorem}
\begin{remark}
    In contrast to the projective situation, one cannot expect that the finite \'etale covering space $X'$ of $X$ in \Cref{main_result} is bimeromorphic to a product $Z\cong B\times E$, see \Cref{example_no_splitting} below.
\end{remark}
\subsection{Outline of the argument}
In the projective case, Hao argues as follows: Let $X$ be a smooth projective variety over $\C$ of Kodaira codimension 1 that admits a holomorphic 1-form $\omega\in H^0(X,\Omega^1_X)$ without zeros. Let $X^{\mathrm{min}}$ be a good minimal model of $X$, and consider the Iitaka fibration $f\colon X^{\mathrm{min}}\rightarrow Y$. By \cite[][Thm. 2.1]{Popa_Schnell_Kodaira_dimension_zeros_1-forms}, the holomorphic 1-form $\omega$ cannot come from $Y$. In particular, the general fiber of $f$ is not contracted by the Albanese morphism. Therefore, it is mapped to a translate of a fixed elliptic curve $E\subset \mathrm{Alb}(X)$. As $\mathrm{Alb}(X)$ is projective, we can dualize the inclusion $E\subset \mathrm{Alb}(X)$ to get a morphism
\begin{equation*}
    \varphi\colon X^{\mathrm{min}}\longrightarrow \mathrm{Alb}(X)\longrightarrow E.
\end{equation*}
The variety $B$ in \Cref{Hao_Kodaira_codim_1}, as well as the splitting $X''\cong B\times E$ of a $\Q$-factorialization $X''\rightarrow X'$ of a finite quasi-\'etale cover $X'\rightarrow X^{\mathrm{min}}$ are then constructed from a general fiber of $\varphi\colon X^{\mathrm{min}}\rightarrow E$. In \cite{Hao_Schreieder_3-folds}, Hao and Schreieder use a similar argument to show that a smooth projective threefold of Kodaira codimension 1 over $\C$, that fibers over the circle $S^1$, admits a nowhere vanishing holomorphic 1-form.\par
Now, let $X$ be a compact connected K\"ahler manifold of Kodaira codimension 1 that admits a holomorphic 1-form $\omega\in H^0(X,\Omega^1_X)$ without zeros. Note that we cannot apply the result by Popa-Schnell \cite[][Thm. 2.1]{Popa_Schnell_Kodaira_dimension_zeros_1-forms}. Moreover, $\mathrm{Alb}(X)$ is not projective, and therefore, the morphism $\varphi$ does not exist in general.\par
To overcome these obstacles, we invoke Lin's construction of an algebraic approximation via tautological models and families \cite{Lin_algebraic_approximation_codim_1} to obtain the following diagram
\[\begin{tikzcd}
	{\mathcal{X}} && {\mathcal{T}/G} \\
	& V && {Y/G.}
	\arrow["\sim", dashed, from=1-1, to=1-3]
	\arrow["\pi"', from=1-1, to=2-2]
	\arrow["{\Pi/G}", from=1-3, to=2-2]
	\arrow["Q/G"', from=1-3, to=2-4]
\end{tikzcd}\]
Here, $\pi\colon\mathcal{X}\rightarrow V$ is a deformation of $X\cong\mathcal{X}_0:=\pi^{-1}(0)$ over a complex vector space $V$, and $\mathcal{T}$ is a complex analytic variety on which a finite group $G$ acts. The morphism
\begin{equation*}
    \Pi\colon\mathcal{T}\longrightarrow Y\times V\longrightarrow V
\end{equation*}
is an explicitly constructed $G$-equivariant family of elliptic fibrations $Q_v\colon\mathcal{T}_v:=\Pi^{-1}(v)\rightarrow Y$ over a fixed smooth projective base $Y$, called the tautological family. The central elliptic fibration
\begin{equation*}
    Q_0/G\colon\mathcal{T}_0/G\longrightarrow Y/G
\end{equation*}
is bimeromorphic to the Iitaka fibration of $X$. Moreover, the deformation $\pi\colon\mathcal{X}\rightarrow V$ has the property that the set of points $v\in V$ such that $\mathcal{X}_v:=\pi^{-1}(v)$ is a projective manifold, is dense. The two families $\pi\colon\mathcal{X}\rightarrow V$ and $\Pi/G\colon\mathcal{T}/G\rightarrow V$ are bimeromorphic over $V$. The elliptic fibrations $Q_v/G\colon\mathcal{T}_v/G\rightarrow Y/G$ are locally isomorphic over $Y/G$, i.e. there is an open cover $\{U_i\}_{i\in I}$ of $Y/G$ such that for every $U_i$ the isomorphism class of the elliptic fibrations
\begin{equation*}
    Q_v\colon\mathcal{T}_v/G\vert_{U_i}\longrightarrow U_i
\end{equation*}
does not depend on $v\in V$.\par
We then show that the Kodaira dimension $\kappa(\mathcal{T}_v/G)$ is constant in $v\in V$ (see \Cref{Kodaira_dim_constanttautological_family}), and that a small deformation of $X$ also admits a holomorphic 1-form without zeros (see \Cref{lemma_deformation}). In particular, there is a point $v_0\in V$ such that $\mathcal{X}_{v_0}$ is a smoooth projective variety of Kodaira codimension 1 that admits a holomorphic 1-form without zeros. As the Iitaka fibration of $\mathcal{X}_{v_0}$ is birational to $Q_{v_0}/G\colon\mathcal{T}_{v_0}/G\rightarrow Y/G$, and the elliptic fibrations $Q_v/G\colon\mathcal{T}_v/G\rightarrow Y/G$ are locally isomorphic over $Y$, we conclude that the general fiber of $Q_0/G\colon\mathcal{T}_0/G\rightarrow Y/G$ is not contracted by the Albanese morphism
\begin{equation*}
    \mathrm{alb}_{\mathcal{T}_0/G}\colon\mathcal{T}_0/G\longrightarrow\mathrm{Alb}(\mathcal{T}_0/G).
\end{equation*}
In particular, a general fiber of $Q_0/G\colon\mathcal{T}_0/G\rightarrow Y$ is isomorphic to a fixed elliptic curve. This overcomes the first obstacle.\par
To finish the argument, construct a finite \'etale cover $Y'\rightarrow Y$ such that the $G$-equivariant elliptic fibration $\mathcal{T}_0\times_YY'\rightarrow Y'$ has trivial monodromy (see\Cref{construction_X'}). We then show that $\mathcal{T}_0\times_YY'\rightarrow Y'$ is a locally trivial elliptic fiber bundle with trivial monodromy. It remains to carefully describe the $G$-action on $\mathcal{T}_0\times_YY'\rightarrow Y'$. Taking a carefully chosen subquotient, yields the desired elliptic fiber bundle $Z\rightarrow B$.
\subsection{Conventions and notation}
By a complex analytic variety, we mean a Hausdorff, second-countable, irreducible, and reduced complex space. A complex analytic variety is called K\"ahler if it admits a K\"ahler metric in the sense of \cite[][II. 1]{Varouchas_Kaehler_spaces_proper_open_morphisms}. On a normal complex analytic variety $X$, we define the canonical sheaf as $\mathcal{K}_X:=i_*\mathcal{K}_{X^{\mathrm{sm}}}$, and the $m$-pluricanonical sheaf as $\mathcal{K}_X^m:=i_*\mathcal{K}_{X^{\mathrm{sm}}}^m$ for all $m\geq 0$, where $i\colon X^{\mathrm{sm}}\hookrightarrow X$ is the inclusion of the smooth locus. A compact complex analytic K\"ahler variety $X$ is called minimal if it is $\Q$-factorial (i.e. every Weil divisor is $\Q$-Cartier and $\mathcal{K}_X^m$ is invertible for some positive integer $m$), has terminal singularities, and the canonical sheaf $\mathcal{K}_X$ is nef. It is furthermore called a good minimal model, if $\mathcal{K}_X^m$ is globally generated for some positive integer $m$. We work in the analytic topology rather than the Zariski topology. In particular, local properties of objects or morphisms are properties which hold true over analytically open subsets. Consider morphisms $X\rightarrow S\leftarrow Y$ between complex analytic varieties. By $X\times_S^{\mathrm{norm}}Y$ we denote the normalization of $X\times_SY$. 
\subsection*{Acknowledgments}
I would like to thank Stefan Schreieder for suggesting this problem to me, for very helpful discussions, and for comments on this manuscript. The research was conducted in the framework of the DFG-funded research training group RTG 2965: From Geometry to Numbers.
\section{Preliminaries}
\subsection{Singular K\"ahler spaces}
A complex analytic variety is called K\"ahler if it admits a K\"ahler metric in the sense of \cite[][II. 1]{Varouchas_Kaehler_spaces_proper_open_morphisms}. We recall some standard results about compact K\"ahler varieties.
\begin{proposition}\cite[][II. Prop. 1.3.1, Cor. 3.2.2]{Varouchas_Kaehler_spaces_proper_open_morphisms}, \cite[][Prop. 3.5]{Graf_Kirschner_finite_quotients_3-tori}
    Let $f\colon X\rightarrow Y$ be a morphism between compact complex analytic varieties. Then the following holds true.
    \begin{enumerate}[label=(\roman*)]
        \item If $Y$ is K\"ahler and $f$ is a closed embedding, then $X$ is K\"ahler.
        \item If $Y$ is K\"ahler and $f$ is a projective morphism, then $X$ is K\"ahler. In particular, by Hironaka's result on resolution of singularities, compact complex analytic K\"ahler varieties admit desingularizations by compact K\"ahler manifolds.
        \item If $f$ is a finite surjective morphism, then $X$ is K\"ahler if and only if $Y$ is K\"ahler.
    \end{enumerate}
\end{proposition}
\begin{definition}
    Let $X$ be a compact K\"ahler manifold. The Kodaira dimension $\kappa(X)$ is defined as
    \begin{equation*}
        \kappa(X):=\dim\mathrm{Proj}\left(\bigoplus\limits_{m\geq 0}H^0(X,\mathcal{K}_X^m)\right).
    \end{equation*}
    If $X$ is a compact normal complex analytic variety that is bimeromorphic to a compact K\"ahler manifold, then the Kodaira dimension $\kappa(X)$ is defined as the Kodaira dimension of $\tilde{X}$, where $\tilde{X}\rightarrow X$ is a resolution of singularities.
\end{definition}
\begin{definition}\cite[][Def. 1.1]{Catanese_quasi_etale}
    Let $X$ and $Y$ be normal complex analytic varieties. A surjective morphism $f\colon X\rightarrow Y$ is called quasi-\'etale if it is quasi-finite and \'etale in codimension 1.
\end{definition}
\begin{remark}\cite[][Rem. 3.1]{Catanese_quasi_etale}
    Let $f\colon X\rightarrow Y$ be a quasi-\'etale morphism between complex analytic varieties. Suppose that $Y$ is smooth and that $X$ is normal. Then $f$ is \'etale. In particular, the branch locus of a quasi-\'etale morphism $f\colon X\rightarrow Y$ is contained in the singular locus of $Y$.
\end{remark}
We need the following lemma, cf. \cite[][Prop. 3.2]{Catanese_quasi_etale}
\begin{lemma}\label{canonical_sing_plus_Cartier_ascends_if_quasi-etale}
    Let $f\colon X\rightarrow Y$ be a quasi-\'etale morphism. Suppose that $Y$ has canonical singularities, and that the canonical sheaf $\mathcal{K}_Y$ is invertible. Then $X$ has canonical singularities and $\mathcal{K}_X\cong f^*\mathcal{K}_Y$. In particular, $\mathcal{K}_X$ is invertible.
\end{lemma}
\begin{proof}
    Let $X^\circ\subset X$, and $Y^\circ$ be open subsets such that $f^\circ\colon X^\circ\rightarrow Y^\circ$ is \'etale, $X^\circ$ and $Y^\circ$ are both smooth, and such that the codimensions of $X\setminus X^\circ$ and of $Y\setminus Y^\circ$ is at least two. As $\mathcal{K}_Y$ is invertible, also $f^*\mathcal{K}_Y$ is invertible. In particular, $\mathcal{K}_X$ and $f^*\mathcal{K}_Y$ are reflexive sheaves isomorphic in codimension 1. Thus, $\mathcal{K}_X\cong f^*\mathcal{K}_Y$. In particular, $\mathcal{K}_X$ is invertible. Therefore, \cite[][Thm. 5.20 (3)]{Kollar_Mori} implies that $X$ has canonical singularities.
\end{proof}
\section{Elliptic fibrations}
If not stated otherwise, the following material can be found in \cite[][1.1]{Nakayama_elliptic-fibrations}.
\begin{definition}
    Let $f\colon X\rightarrow Y$ be a proper surjective morphism between normal complex analytic varieties.
    \begin{enumerate}[label=(\roman*)]
        \item The morphism $f$ is called an elliptic fibration if it has connected fibers and the general fiber is a smooth curve of genus 1.
        \item The discriminant locus of an elliptic fibration is the (not necessarily reduced) closed complex subspace $\Delta\subset Y$ over which $f$ is not smooth. The smooth locus is denoted by $Y^*:=Y\setminus\Delta$.
        \item The elliptic fibration $f\colon X\rightarrow Y$ is called projective, if there exists an $f$-ample line bundle on $X$. Similarly, $f$ is called locally projective, if there exists an open cover $Y=\bigcup U_i$ such that the elliptic fibrations $f^{-1}(U_i)\rightarrow U_i$ are projective for each $U_i$.
    \end{enumerate}
\end{definition}
In the following, we will only consider elliptic fibrations $f\colon X\rightarrow Y$ over a connected complex manifold $Y$.
\subsection{Jacobian fibrations}
\begin{definition}\label{def_variation_HS_Jacobian_fibration}
    Let $f\colon X\rightarrow Y$ be a smooth elliptic fibration.
    \begin{enumerate}[label=(\roman*)]
        \item The variation of Hodge structures $H$ associated to $f$ consists of the following data: the local system $H_\Z:=R^{1}f_*\Z$ on $Y$ of free abelian groups of rank $2$, the associated vector bundle $\mathcal{H}:=H_\Z\otimes\OO_Y$ on $Y$, the holomorphic subbundle $\mathcal{F}^{1}:=f_*\Omega^1_{X/Y}$, and the quotient $\mathcal{H}^{0,1}:=\mathcal{H}/\mathcal{F}^{1}$.
        \item We say that $f$ has trivial monodromy if the local system $H_\Z$ is isomorphic to $\Z^{2}$.
        \item We define the Jacobian fibration associated to $f$ as the complex manifold
        \begin{equation*}
            J_H:=\mathcal{H}^{0,1}/H_\Z
        \end{equation*}
        obtained by taking the quotient of the vector bundle $\mathcal{H}^{0,1}\rightarrow Y$ by the $H_\Z$-action together with the natural morphism $\pi\colon J_H\rightarrow Y$.
        \item The zero section of $\mathcal{H}^{0,1}\rightarrow Y$ induces the global section $\sigma_0\colon Y\rightarrow J_H$. The sheaf of local sections of $\pi\colon J_H\rightarrow Y$ is a sheaf of abelian groups with addition of local sections and neutral element $\sigma_0$. It is denoted by $\mathcal{J}_H=\mathcal{H}^{0,1}/H_\Z$, where we identify the vector bundle $\mathcal{H}^{0,1}$ with its sheaf of local sections.
        \item Let $\eta\in H^1(Y,\mathcal{J}_H)$ be a cohomology class. The $\mathcal{J}_H$-torsor associated to $\eta$ is denoted by
        \begin{equation*}
            \pi^\eta\colon J^\eta_H\rightarrow Y.
        \end{equation*}
        \item We denote by $\mathrm{exp}\colon H^1(Y,\mathcal{H}^{0,1})\rightarrow H^1(Y,\mathcal{J}_H)$ and $c\colon H^1(Y,\mathcal{J}_H)\rightarrow H^2(Y,H_\Z)$ the maps induced by the short exact sequence $0\rightarrow H_\Z\rightarrow\mathcal{H}^{0,1}\rightarrow\mathcal{J}_H\rightarrow 0$, and call them the exponential, and Chern class map, respectively.
    \end{enumerate}
\end{definition}
\begin{corollary}\label{corollary_isotrivial_and_trivial_vhs}
    Let $f\colon X\rightarrow Y$ be a smooth elliptic fibration. Suppose that $f$ is isotrivial with fiber $E\cong\C/\Lambda$ and that its monodromy is trivial. Then $J_H=Y\times E$ and $\mathcal{J}_H\cong\OO_Y/\Lambda$.
\end{corollary}
\begin{definition}\label{definition_isotrivial_and_trivial_vhs}
    In the situation of \Cref{corollary_isotrivial_and_trivial_vhs}, we denote the $\mathcal{J}_H$ torsor associated to $\eta\in H^1(Y,\mathcal{J}_H)$ by $\pi^\eta\colon(Y\times E)^\eta\rightarrow Y$.
\end{definition}
The following proposition builds on \cite[][Thm. 3.20, Prop. 3.23]{Claudon_Hoering_Lin_fundamental_group} and \cite[][Prop. 1.3.3]{Nakayama_elliptic-fibrations}. See also \cite[][Prop. 2.5]{Claudon_tori_fiber_bundles}.
\begin{proposition}\label{fundamental_properties_special_torus_bundles}\cite[][Prop. 2.7]{Pietig_1-forms_Kaehler_threefolds}
    Let $Y$ be a compact connected K\"ahler manifold, let $E=\C/\Lambda$ be an elliptic curve, and let $\eta\in H^1(Y,\OO_Y/\Lambda)$.
    \begin{enumerate}[label=(\roman*)]
        \item\label{torus_bundle_kaehler} The complex manifold $(Y\times E)^\eta$ is K\"ahler if and only if $c(\eta)\in H^2(Y,\Lambda)$ is torsion.
        \item\label{1-forms_on_torus_bundles} Suppose that $(Y\times E)^\eta$ is K\"ahler. Let $\pi\colon (Y\times E)^\eta\rightarrow Y$ be the projection. Then for every $y\in Y$ the restriction map
        \begin{equation*}
            H^0((Y\times E)^\eta,\Omega^1_{(Y\times E)^\eta})\longrightarrow H^0(\pi^{-1}(y),\Omega^1_{\pi^{-1}(y)})\cong H^0(E,\Omega^1_E)
        \end{equation*}
        fits into a short exact sequence
        \begin{equation*}
            0\longrightarrow H^0(Y,\Omega^1_Y)\longrightarrow H^0((Y\times E)^\eta,\Omega^1_{(Y\times E)^\eta})\longrightarrow H^0(E,\Omega^1_E)\longrightarrow 0.
        \end{equation*}
        In particular, any holomorphic 1-form $\omega\in H^0((Y\times E)^\eta,\Omega^1_{(Y\times E)^\eta})$ that does not come from $Y$ has no zeros.
        \item\label{torus_bundle_multisection} The following is equivalent:
        \begin{enumerate}[label=(\alph*)]
            \item\label{torus_bundle_multisection_a} The smooth elliptic fibration $(Y\times E)^\eta\rightarrow Y$ admits a multisection.
            \item\label{torus_bundle_multisection_b} There is a finite \'etale cover $Y'\rightarrow Y$ such that $(Y\times E)^\eta\times_YY'\cong Y'\times E$.
            \item\label{torus_bundle_multisection_c} The cohomology class $\eta\in H^1(Y,\OO_Y/\Lambda)$ is torsion.
        \end{enumerate}
    \end{enumerate}
\end{proposition}
\begin{remark}\label{example_no_splitting}
    \Cref{fundamental_properties_special_torus_bundles} allows us to construct a compact K\"ahler manifold $X$ of Kodaira dimension $\kappa(X)=\dim X-1$ that admits a holomorphic 1-form without zeros such that no finite \'etale cover of $X$ splits into a product where one factor is an elliptic curve. This shows that \Cref{Hao_Kodaira_codim_1} fails in the K\"ahler case. Indeed, let $C$ be a smooth projective curve of genus $g\geq 2$, and let $E=\C/\Lambda$ be an elliptic curve. Consider the exact sequence
    \begin{equation*}
        H^1(C,\Lambda)\longrightarrow H^1(C,\OO_C)\stackrel{\mathrm{exp}}{\longrightarrow}H^1(C,\OO_C/\Lambda)\stackrel{c}{\longrightarrow}H^2(C,\Lambda),
    \end{equation*}
    and choose a class $\alpha\in H^1(C,\OO_C)$ that is not in the image of
    \begin{equation*}
        H^1(C,\Lambda)\otimes\Q\longrightarrow H^1(C,\OO_C).
    \end{equation*}
    Since $H^1(C,\OO_C)$ is a non-trivial complex vector space, and $H^1(C,\Lambda)$ is a finitely generated abelian group, such an $\alpha$ exists. Define $\eta:=\mathrm{exp}(\alpha)$. Note that the choice of $\alpha$ implies that $\eta$ is non-torsion. The exactness of the above sequence implies that the Chern class $c(\eta)$ vanishes. Thus, by \Cref{fundamental_properties_special_torus_bundles} \cref{torus_bundle_kaehler} and \cref{1-forms_on_torus_bundles}, the space $X:=(C\times E)^\eta$ is a compact K\"ahler manifold that admits a holomorphic 1-form without zeros. Since $C$ is of genus $g\geq 2$, $\kappa(X)=\dim C=\dim X-1$. We claim that no finite \'etale cover of $X$ splits into a product where one factor is an elliptic curve. To show this, suppose that a finite \'etale cover $X'\rightarrow X$ splits into a product $X'\cong D\times F$, where $F$ is an elliptic curve. Since $g(C)\geq 2$, the composition
    \begin{equation*}
        \{x\}\times F\longhookrightarrow D\times F\longrightarrow X\longrightarrow C
    \end{equation*}
    must be constant for every $x\in D$, and hence there exists a morphism $D\rightarrow C$ such that the composition $D\times F\rightarrow X\rightarrow C$ factors through $D\times F\rightarrow D\rightarrow C$. Since $D\times F\rightarrow X\rightarrow C$ is surjective, the morphism $D\rightarrow C$ cannot be constant. Therefore, the image $S\subset X$ of $D\times\{0\}\rightarrow X$ defines a multisection of $X\rightarrow C$. Thus, by \Cref{fundamental_properties_special_torus_bundles} \cref{torus_bundle_multisection} \cref{torus_bundle_multisection_a}, the cohomology class $\eta$ is torsion, which is a contradiction.
\end{remark}
\subsubsection{Equivariant elliptic fiber bundles with trivial monodromy}
\begin{proposition}\label{trivial_vhs_constant_j_all_invariant}\cite[][Prop. 2.22]{Pietig_1-forms_Kaehler_threefolds}
    Let $f\colon X\rightarrow Y$ be an elliptic fibration between compact connected complex manifolds, which has local meromorphic sections over every point of $Y$ and is smooth outside a normal crossing divisor $\Delta\subset Y$. Suppose that a finite group $G$ acts on $X$ and $Y$ such that $f$ is $G$-equivariant. Suppose that the local system $\Lambda:=R^1f_*\Z\vert_{Y^*}$ of the variation of Hodge structures induced by $f$ is $G$-equivariantly isomorphic to $\Z^2$, where $\Z^2$ is endowed with the trivial $G$-action, and that the smooth fibers of $f$ are isomorphic to $\C/\Lambda$. Then the following holds true.
    \begin{enumerate}[label=(\roman*)]
        \item There is a cohomology class $\eta\in H^1(Y,\OO_Y/\Lambda)^G$ such that $f\colon X\rightarrow Y$ is bimeromorphic to $\pi^\eta\colon(Y\times E)^\eta\rightarrow Y$. Moreover, there is a $G$-action on $(Y\times E)^\eta$ such that $f\colon X\rightarrow Y$ is $G$-equivariantly bimeromorphic to $\pi^\eta\colon(Y\times E)^\eta\rightarrow Y$.
        \item For each $g\in G$ denote by $(Y\times E)^\eta\times_gY$ the fiber product
        \[\begin{tikzcd}
	       {(Y\times E)^\eta\times_gY} & {(Y\times E)^\eta} \\
	       Y & {Y.}
	       \arrow["\psi_g", from=1-1, to=1-2]
	       \arrow[from=1-1, to=2-1]
	       \arrow[from=1-2, to=2-2]
	       \arrow["g", from=2-1, to=2-2]
        \end{tikzcd}\]
        Then there are isomorphisms
        \begin{equation*}
            \varphi_g\colon (Y\times E)^\eta\times_gY\stackrel{\sim}{\longrightarrow} (Y\times E)^\eta
        \end{equation*}
        over $Y$ such that $\varphi_h\circ(\varphi_g\times_hY)=\varphi_{gh}$ for all $g$, $h\in G$. The isomorphisms $\varphi_g$ only depend on the cohomology class $\eta\in H^1(Y,\OO_Y/\Lambda)^G$, and not on the $G$-action on $(Y\times E)^\eta$.
        \item The composition in the upper row of
        \[\begin{tikzcd}
	       {(Y\times E)^\eta} & {(Y\times E)^\eta\times_gY} & {(Y\times E)^\eta} & {(Y\times E)^\eta} \\
	       Y & Y & Y & {Y,}
	       \arrow["{\varphi^{-1}_g}", from=1-1, to=1-2]
	       \arrow[from=1-1, to=2-1]
	       \arrow["{\psi_g}", from=1-2, to=1-3]
	       \arrow[from=1-2, to=2-2]
	       \arrow["{g^{-1}}", from=1-3, to=1-4]
	       \arrow[from=1-3, to=2-3]
	       \arrow[from=1-4, to=2-4]
	       \arrow["\id", from=2-1, to=2-2]
	       \arrow["g", from=2-2, to=2-3]
	       \arrow["{g^{-1}}", from=2-3, to=2-4]
        \end{tikzcd}\]
        equals translation on each fiber by an element $\sigma(g)\in E$. The induced map $\sigma\colon G\rightarrow E$ is a group homomorphism, which determines the $G$-action on $(Y\times E)^\eta$.
        \item If $X$ and $Y$ are K\"ahler, then $(Y\times E)^\eta$ is also K\"ahler.
    \end{enumerate}
\end{proposition}
\subsection{Equivariant Weierstraß models}\label{section_Weierstrass_models}
\begin{definition}\label{definition_Weierstrass_models}\cite[][Def. 1.1, 2]{Nakayama_Weierstrass}, \cite[][3E]{Claudon_Hoering_Lin_fundamental_group}
    Let $Y$ be a connected complex manifold and let $\mathcal{L}$ be a line bundle on $Y$. Choose global sections $\alpha\in H^0(Y,\mathcal{L}^{-4})$ and $\beta\in H^0(Y,\mathcal{L}^{-6})$ such that $4\alpha^3+27\beta^2\neq 0$ in $H^0(Y,\mathcal{L}^{-12})$.
    \begin{enumerate}[label=(\roman*)]
        \item The Weierstraß model $p\colon W(\mathcal{L},\alpha,\beta)\rightarrow Y$ is then defined as
        \begin{equation*}
            W(\mathcal{L},\alpha,\beta):=\{y^2z-(x^3+\alpha xz^2+\beta z^3)=0\}\subset\mathbb{P}:=\mathbb{P}(\OO_Y\oplus\mathcal{L}^2\oplus\mathcal{L}^3)\longrightarrow Y,
        \end{equation*}
        where $x\in H^0(\mathbb{P},\OO_\mathbb{P}(1)\otimes q^*\mathcal{L}^{-2})$, $y\in H^0(\mathbb{P},\OO_\mathbb{P}(1)\otimes q^*\mathcal{L}^{-3})$, $z\in H^0(\mathbb{P},\OO_\mathbb{P}(1))$ are given by the direct sum inclusions $\mathcal{L}^2\hookrightarrow\OO_Y\oplus\mathcal{L}^2\oplus\mathcal{L}^3$, $\mathcal{L}^3\hookrightarrow\OO_Y\oplus\mathcal{L}^2\oplus\mathcal{L}^3$, $\OO_Y\hookrightarrow\OO_Y\oplus\mathcal{L}^2\oplus\mathcal{L}^3$.
        \item The subspace $W(\mathcal{L},\alpha,\beta)^\#\subset W(\mathcal{L},\alpha,\beta)$ is defined as the locus where $p$ is smooth.
        \item The canonical section of $W(\mathcal{L},\alpha,\beta)\rightarrow Y$ is defined as $\Sigma:=\{x=z=0\}$. The sheaf of local meromorphic sections of $p\colon W(\mathcal{L},\alpha,\beta)\rightarrow Y$ is denoted by $\mathcal{W}(\mathcal{L},\alpha,\beta)^{\mathrm{mer}}$ and the sheaf of local sections with image in $W(\mathcal{L},\alpha,\beta)^\#$ is denoted by $\mathcal{W}(\mathcal{L},\alpha,\beta)^\#$. Both sheaves are sheaves of abelian groups with addition of local (meromorphic) sections with respect to $\Sigma$.
        \item The Weierstraß model $p\colon W(\mathcal{L},\alpha,\beta)\rightarrow Y$ is called minimal if there is no prime divisor $\Delta$ on $Y$ such that $\mathrm{div}(\alpha)-4\Delta$ and $\mathrm{div}(\beta)-6\Delta$ are both effective divisors on $Y$.
        \item Let $G$ be a finite group. The Weierstraß model $p\colon W(\mathcal{L},\alpha,\beta)\rightarrow Y$ is called $G$-equivariant if $G$ acts on $W(\mathcal{L},\alpha,\beta)$ and on $Y$ such that $p$ is $G$-equivariant and the canonical section $\Sigma$ is fixed by the $G$-action.
    \end{enumerate}
    In the case where $\mathcal{L}$, $\alpha$, and $\beta$ are evident from the context, we omit these symbols in the notation.
\end{definition}
\begin{remark}\label{facts_Weierstrass_modles}
    The morphism $p\colon W(\mathcal{L},\alpha,\beta)\rightarrow Y$ is projective, flat, and surjective. Its discriminant locus is given by $\mathrm{div}(4\alpha^3+27\beta^2)$.\par
    The total space $W(\mathcal{L},\alpha,\beta)$ is normal \cite[][1.2]{Nakayama_Weierstrass}. If $p\colon W(\mathcal{L},\alpha,\beta)\rightarrow Y$ is a minimal Weierstraß model, and if $\mathrm{div}(4\alpha^3+27\beta^2)$ is a normal crossing divisor, then $W(\mathcal{L},\alpha,\beta)$ has rational singularities \cite[][Cor. 2.4]{Nakayama_elliptic-fibrations}.\par
    Moreover, there is an isomorphism $R^1p_*\OO_{W(\mathcal{L},\alpha,\beta)}\cong\mathcal{L}$ \cite[][1.2]{Nakayama_Weierstrass}. Suppose that a finite group $G$ acts on $W(\mathcal{L},\alpha,\beta)$ and on $Y$ such that the Weierstraß model $p\colon W(\mathcal{L},\alpha,\beta)\rightarrow Y$ is $G$-equivariant. Then we can apply the flat base change theorem \cite[][III. Prop. 9.3]{Hartshorne_AG} to $R^1p_*\OO_{W(\mathcal{L},\alpha,\beta)}\cong\mathcal{L}$ to get a $G^{op}$-action on $\mathcal{L}$.
\end{remark}
\begin{theorem}\label{bimeromorphic_classification_elliptic_fibration_with_section}\cite[][Thm. 2.5]{Nakayama_Weierstrass}
    Let $f\colon X\rightarrow Y$ be an elliptic fibration over a connected complex manifold $Y$. Suppose that $f$ admits a meromorphic section. Then there exists a unique minimal Weierstraß model $p\colon W(\mathcal{L},\alpha,\beta)\rightarrow Y$ that is bimeromorphic to $f$.
\end{theorem}
\begin{definition}\cite[][3E]{Claudon_Hoering_Lin_fundamental_group}
    Let $Y$ be a connected complex manifold and let $G$ be a finite group acting on $Y$. Let $\Delta\subset Y$ be a $G$-equivariant normal crossing divisor and let $H$ be a $G$-equivariant variation of Hodge structures on $j\colon Y^*:=Y\setminus\Delta\hookrightarrow Y$ associated to some elliptic fiber bundle over $Y^*$. We define
    \begin{equation*}
        \mathcal{E}_G(Y,\Delta,H)
    \end{equation*}
    to be the set of equivalence classes of $G$-equivariant elliptic fibrations $f\colon X\rightarrow Y$ that satisfy the following properties:
    \begin{enumerate}[label=(\roman*)]
        \item $f\colon X\rightarrow Y$ has local meromorphic sections over every point of $Y$,
        \item the elliptic fibration $f^{-1}(Y^*)\rightarrow Y^*$ is bimeromorphic to an elliptic fiber bundle over $Y^*$, and
        \item this elliptic fibration induces $H$ as its $G$-equivariant variation of Hodge structures.
    \end{enumerate}
    Two such elliptic fibrations are equivalent if they are $G$-equivariantly bimeromorphic.
\end{definition}
\begin{lemma}\label{existence_unique_eq_Weierstrass_model}\cite[][Cor. 2.7]{Nakayama_Weierstrass}
    Let $Y$ be a connected complex manifold and let $G$ be a finite group acting on $Y$. Let $\Delta\subset Y$ be a $G$-equivariant normal crossing divisor and let $H$ be a $G$-equivariant variation of Hodge structures on $j\colon Y^*:=Y\setminus\Delta\hookrightarrow Y$ associated to some elliptic fiber bundle over $Y^*$. Then there exists a unique minimal Weierstraß model $p\colon W(\mathcal{L},\alpha,\beta)\rightarrow Y$ that is smooth over $Y^*$ and induces $H$ as its $G$-equivariant variation of Hodge structures.
\end{lemma}
We can construct elements of $\mathcal{E}_G(Y,\Delta,H)$ as follows: Let $G$ be a finite group and let $p\colon W(\mathcal{L},\alpha,\beta)\rightarrow Y$ be a minimal $G$-equivariant Weierstraß model that is smooth over $Y^*$ and has $H$ as its associated variation of Hodge structures, which exists by \Cref{existence_unique_eq_Weierstrass_model}. The subspace $W(\mathcal{L},\alpha,\beta)^\#\subset W(\mathcal{L},\alpha,\beta)$ defines a complex analytic group variety over $Y$ with zero section $\Sigma\subset W(\mathcal{L},\alpha,\beta)^\#$, where the group structure is given by the addition of local sections with respect to $\Sigma$. The action of $W(\mathcal{L},\alpha,\beta)^\#$ on $W(\mathcal{L},\alpha,\beta)^\#$ can be extended to an action of $W(\mathcal{L},\alpha,\beta)^\#$ on $W(\mathcal{L},\alpha,\beta)$ \cite[][5.1.1]{Nakayama_global_elliptic_fibration}. By gluing local patches of $p\colon W(\mathcal{L},\alpha,\beta)\rightarrow Y$ along a cocycle of local sections representing the cohomology class $\eta_G\in H_G^1(Y,\mathcal{W}(\mathcal{L},\alpha,\beta)^\#)$, one can thus construct the twisted $G$-equivariant Weierstraß model $p^{\eta_G}\colon W(\mathcal{L},\alpha,\beta)^{\eta_G}\rightarrow Y$ \cite[][3E]{Claudon_Hoering_Lin_fundamental_group}. This construction defines a map
\begin{equation}\label{twisted_Weierstrass_models}
    H_G^1(Y,\mathcal{W}(\mathcal{L},\alpha,\beta)^\#)\longrightarrow\mathcal{E}_G(Y,\Delta,H).
\end{equation}
\begin{lemma}\label{Weierstrass_model_canonical_sing_canonical_bundle}
    Let $p\colon W(\mathcal{L},\alpha,\beta)\rightarrow Y$ be a $G$-equivariant minimal Weierstraß model over $Y$. Suppose that the discriminant divisor $\mathrm{div}(4\alpha^3+27\beta^2)$ has normal crossings. Let $\eta\in H_G^1(Y,\mathcal{W}^\#)$ be a cohomology class. Then the following hold true.
    \begin{enumerate}[label=(\roman*)]
        \item The compact complex analytic variety $W(\mathcal{L},\alpha,\beta)^\eta$ has canonical singularities.
        \item The canonical sheaf $\mathcal{K}_{W(\mathcal{L},\alpha,\beta)}$ is invertible, and there is a $G$-equivariant isomorphism
        \begin{equation*}
            \varphi\colon\mathcal{K}_{W(\mathcal{L},\alpha,\beta)}\stackrel{\sim}{\longrightarrow} p^*(\mathcal{K}_Y\otimes\mathcal{L}^\vee).
        \end{equation*}
        \item\label{Wmodel_canonical_bundle_twisted_locally_compatible} Choose a $G$-invariant open cover $\{U_i\}_{i\in I}$ of $Y$ over which $p^\eta\colon W(\mathcal{L},\alpha,\beta)^\eta\rightarrow Y$ is isomorphic to $p\colon W(\mathcal{L},\alpha,\beta)\rightarrow Y$. Fix isomorphisms
        \begin{equation*}
            \gamma_i\colon (p^\eta)^{-1}(U_i)\stackrel{\sim}{\longrightarrow} p^{-1}(U_i)
        \end{equation*}
        over $U_i$ for all $i\in I$. Then there is a $G$-equivariant isomorphism
        \begin{equation*}
            \varphi^\eta\colon\mathcal{K}_{W(\mathcal{L},\alpha,\beta)^\eta}\stackrel{\sim}{\longrightarrow}(p^\eta)^*(\mathcal{K}_Y\otimes\mathcal{L}^\vee)
        \end{equation*}
        such that the composition
        \begin{align*}
            \mathcal{K}_{W(\mathcal{L},\alpha,\beta)}\vert_{p^{-1}(U_i)}&=(\gamma_i^{-1})^*\gamma_i^*\mathcal{K}_{W(\mathcal{L},\alpha,\beta)}\vert_{p^{-1}(U_i)}\\
            &\stackrel{d\gamma_i}{\longrightarrow}\gamma_i^*\mathcal{K}_{W(\mathcal{L},\alpha,\beta)^\eta}\vert_{p^{-1}(U_i)}\\
            &\stackrel{\gamma_i^*\varphi^\eta}{\longrightarrow}\gamma_i^*(p^\eta)^*(\mathcal{K}_Y\otimes\mathcal{L}^\vee)\vert_{p^{-1}(U_i)}\\
            &=p^*(\mathcal{K}_Y\otimes\mathcal{L}^\vee)\vert_{p^{-1}(U_i)}
        \end{align*}
        equals $\varphi\vert_{p^{-1}(U_i)}$ for all $i\in I$.
    \end{enumerate}
\end{lemma}
\begin{proof}
    By \Cref{facts_Weierstrass_modles}, the compact complex analytic variety $W$ is normal with rational singularities. Since the twisted Weierstraß model $W^\eta$ is locally over $Y$ isomorphic to $W$, the compact complex analytic variety $W^\eta$ is also normal with rational singularities. Thus, by \cite[][Cor. 5.24]{Kollar_Mori}, to prove that $W(\mathcal{L},\alpha,\beta)^\eta$ has canonical singularities, it suffices to show that $\mathcal{K}_{W(\mathcal{L},\alpha,\beta)^\eta}$ is invertible.\par
    Denote by $W^{\eta,\#}\subset W^\eta$ the smooth locus of $p^\eta\colon W^\eta\rightarrow Y$, and denote by $p^{\eta,\#}\colon W^{\eta,\#}\rightarrow Y$ the restriction of $p^\eta$ to $W^{\eta,\#}$. Then
    \begin{equation*}
        \mathcal{K}_{W^{\eta,\#}}\cong (p^{\eta,\#})^*\mathcal{K}_Y\otimes\omega_{W^{\eta,\#}/Y},
    \end{equation*}
    where $\omega_{W^{\eta,\#}/Y}$ is the sheaf of relative K\"ahler differentials. By \cite[][Lem. 5.1.1]{Nakayama_global_elliptic_fibration}, there is a $G$-equivariant isomorphism
    \begin{equation*}
        \omega_{W/Y}\cong p^*\mathcal{L}^\vee.
    \end{equation*}
    Thus, we obtain an isomorphism
    \begin{equation*}
        \varphi\colon\mathcal{K}_W\stackrel{\sim}{\longrightarrow} p^*(\mathcal{K}_Y\otimes\mathcal{L}^\vee).
    \end{equation*}
    Recall that $W^\eta$ is locally over $Y$ isomorphic to $W$, and that the gluing automorphisms are translations by local sections with image in $W^\#$. Note that translations by local sections fix the relative canonical bundle and $p^*\mathcal{L}^\vee$. Choose an open cover $\{U_i\}_{i\in I}$ of $Y$ as in \Cref{Weierstrass_model_canonical_sing_canonical_bundle} \cref{Wmodel_canonical_bundle_twisted_locally_compatible}. We can restrict the isomorphism $\varphi\colon\mathcal{K}_W\stackrel{\sim}{\rightarrow} p^*(\mathcal{K}_Y\otimes\mathcal{L}^\vee)$ to local patches $W\vert_{U_i}\rightarrow U_i$ and glue these local patches along translation by local sections to obtain the $G$-equivariant isomorphism
    \begin{equation*}
        \varphi^\eta\colon\mathcal{K}_{W^\eta}\stackrel{\sim}{\longrightarrow} (p^\eta)^*(\mathcal{K}_Y\otimes\mathcal{L}^\vee).
    \end{equation*}
    In particular, $\mathcal{K}_{W^\eta}$ is invertible.
\end{proof}
\begin{proposition}\label{bimeromorphic_classification_elliptic_fibrations}\cite[][Prop. 2.10]{Nakayama_Weierstrass}\cite[][Prop. 5.5.1]{Nakayama_global_elliptic_fibration}\cite[][3E]{Claudon_Hoering_Lin_fundamental_group}
    Let $p\colon W(\mathcal{L},\alpha,\beta)\rightarrow Y$ be a minimal $G$-equivariant Weierstraß model, such that the discriminant locus $\Delta:=\mathrm{div}(4\alpha^3+27\beta^2)$ is a normal crossing divisor. Denote by $j\colon Y^*\subset Y$ the inclusion and let $H:=R^1p_*\Z\vert_{Y^*}$ be the associated variation of Hodge structures.
    \begin{enumerate}[label=(\roman*)]
        \item There is a $G$-equivariant short exact sequence
        \begin{equation}\label{ses_vhs_line_bundle_W}
            0\longrightarrow j_*H_\Z\longrightarrow\mathcal{L}\stackrel{\exp}{\longrightarrow}\mathcal{W}^\#\longrightarrow 0.
        \end{equation}
        \item There is an injective map
        \begin{equation}\label{injection_into_H1(Wmer)}
            \mathcal{E}_G(Y,\Delta,H)\longhookrightarrow H_G^1(Y,\mathcal{W}^{\mathrm{mer}})
        \end{equation}
        such that the composition
        \begin{equation}\label{composition_construction_injection_into_H1(Wmer)}
            H_G^1(Y,\mathcal{W}(\mathcal{L},\alpha,\beta)^\#)\longrightarrow\mathcal{E}_G(Y,\Delta,H)\longhookrightarrow H_G^1(Y,\mathcal{W}^{\mathrm{mer}})
        \end{equation}
        equals the map in cohomology induced by the inclusion $\mathcal{W}^\#\subset\mathcal{W}^{\mathrm{mer}}$.
    \end{enumerate}
\end{proposition}
\subsection{Equivariant tautological models}
There are classes in $\mathcal{E}_G(Y,\Delta,H)$ that cannot be represented by a twisted Weierstraß model, the reason being that the map $H^1(Y,\mathcal{W}^\#)\rightarrow H^1(Y,\mathcal{W}^{\mathrm{mer}})$ may not be surjective. Lin introduced the notion of a tautological model \cite[][3.4]{Lin_algebraic_approximation_codim_1}, which are finite covers of Weierstraß models and showed that every class in $\mathcal{E}_G(Y,\Delta,H)$ that comes from an elliptic fibration $f\colon X\rightarrow Y$, where $X$ is bimeromorphic to a compact K\"ahler manifold, can be represented by a tautological model. Moreover, tautological models fit into a family.
\begin{proposition}\label{tautological_models}\cite[][Lem. 3.19, Lem. 3.24]{Claudon_Hoering_Lin_fundamental_group}\cite[][3.4]{Lin_algebraic_approximation_codim_1}
    Let $f\colon X\rightarrow Y$ be a $G$-equivariant elliptic fibration from a compact complex analytic variety $X$ bimeromorphic to a compact K\"ahler manifold to a connected complex manifold $Y$. Suppose that the discriminant divisor of $f$ is a normal crossing divisor and that $f$ admits local meromorphic sections over every point of $Y$. Denote by $H:=R^1f_*\Z\vert_{Y^*}$ the variation of Hodge structures induced by $f$ and denote by $p\colon W(\mathcal{L},\alpha,\beta)\rightarrow Y$ the unique $G$-equivariant minimal Weierstraß model associated to $(Y,\Delta,H)$ by \Cref{existence_unique_eq_Weierstrass_model}.
    \begin{enumerate}[label=(\roman*)]
        \item\label{multiple_of_eta_lifts} Let $\eta_G\in H^1_G(Y,\mathcal{W}^{\mathrm{mer}})$ be the image of the class of $f$ under the map $\mathcal{E}_G(Y,\Delta,H)\rightarrow H^1_G(Y,\mathcal{W}^{\mathrm{mer}})$ from \cref{injection_into_H1(Wmer)} above. Then there is a positive integer $m$ such that $m\eta_G$ can be lifted to a class $\eta'_G\in H^1_G(Y,\mathcal{W}^\#)$.
        \item Fix a positive integer $m$ and a lift $\eta'_G\in H_G^1(Y,\mathcal{W}^\#)$ of $m\eta_G$ as in \cref{multiple_of_eta_lifts}. Then there is a normal complex analytic variety $T$ together with an elliptic fibration $g\colon T\rightarrow Y$, and a finite morphism $\mathrm{mult}_m\colon T\rightarrow W(\mathcal{L},\alpha,\beta)^{\eta'_G}$ of degree $m^2$ such that the following diagram commutes
        \[\begin{tikzcd}
	       {T} && {W(\mathcal{L},\alpha,\beta)^{\eta'_G}} \\
	       & {Y} && {.}
	       \arrow["{\mathrm{mult}_m}", from=1-1, to=1-3]
	       \arrow["q"', from=1-1, to=2-2]
	       \arrow["p^{\eta'_G}", from=1-3, to=2-2]
        \end{tikzcd}\]
        Moreover, the group $G$ acts on $T$ such that the diagram becomes $G$-equivariant. The $G$-equivairant elliptic fibration $q\colon T\rightarrow Y$ is $G$-equivariantly bimeromorphic to $f\colon X\rightarrow Y$. The discriminant locus of $q$ coincides with the discriminant divisor of $p$. On the smooth locus of $q$ and $p$, the morphism $\mathrm{mult}_m$ is the multiplication by $m$-map.
    \end{enumerate}
\end{proposition}
\begin{definition}\label{def_tautological_model}
    In the situation and notation of \Cref{tautological_models}, we call $q\colon T\rightarrow Y$ a $G$-equivariant tautological model associated to $f\colon X\rightarrow Y$, resp. to $\eta_G\in H_G^1(Y,\mathcal{W}^{\mathrm{mer}})$ and $m$.
\end{definition}
\begin{remark}\label{tautological_models_are_twisted_Wmodels_in_good_situation}
    Suppose that the triple $(\mathcal{L},\alpha,\beta)$ can be $G$-equivariantly identified with $(\OO_Y,a,b)$, where $a$, $b\in\C$. Then the Weierstraß model is given by
    \begin{equation*}
        p\colon Y\times E\longrightarrow Y,
    \end{equation*}
    where $E=\{y^2z-(x^3+axz^2+bz^3)=0\}\subset\mathbb{P}^2$. By \cite[][Lem. 1.3.5]{Nakayama_elliptic-fibrations}, every meromorphic section of $Y\times E\rightarrow Y$ is holomorphic. Thus, the sheaves $\mathcal{W}(\tau^*\mathcal{L},\tau^*\alpha,\tau^*\beta)^\#$ and $\mathcal{W}(\tau^*\mathcal{L},\tau^*\alpha,\tau^*\beta)^{\mathrm{mer}}$ agree, and, by \Cref{corollary_isotrivial_and_trivial_vhs}, they are isomorphic to $\OO_Y/\Lambda$, where $\Lambda\subset\C$ is a lattice such that $E\cong\C/\Lambda$. In particular, the composition in \cref{composition_construction_injection_into_H1(Wmer)} is of the form
    \begin{equation*}
        \id\colon H^1_G(Y,\OO_Y/\Lambda)\longrightarrow\mathcal{E}_G(Y,\varnothing,\Lambda)\longhookrightarrow H^1_G(Y,\OO_Y/\Lambda).
    \end{equation*}
    Let $\eta\in H_G^1(Y,\OO_Y/\Lambda)$. As the multiplication by $m$-map induces a morphism
    \begin{equation*}
        \mathrm{mult}_m\colon (Y\times E)^\eta\longrightarrow(Y\times E)^{m\eta},
    \end{equation*}
    every $G$-equivariant tautological model in this situation is a twisted $G$-equivariant Weierstraß model.
\end{remark}
\begin{lemma}\label{properties_tautological_models}
    In the situation and notation of \Cref{tautological_models},
    \begin{enumerate}[label=(\roman*)]
        \item The morphism $\mathrm{mult}_m\colon T\rightarrow W(\mathcal{L},\alpha,\beta)^{\eta'_G}$ is quasi-\'etale.
        \item The compact complex analytic variety $T$ has canonical singularities.
        \item The canonical sheaf of $T$ is invertible. Moreover, there is a $G$-equivariant isomorphism $\mathcal{K}_T\cong q^*(\mathcal{K}_Y\otimes\mathcal{L}^\vee)$.
    \end{enumerate}
\end{lemma}
\begin{proof}
    Since on the smooth loci of $q$ and $p$, the morphism $\mathrm{mult}_m$ is the multiplication by $m$-map, it is unramified in codimension 2, i.e. quasi-\'etale.\par
    Recall that the elliptic fibration $f\colon X\rightarrow Y$ from \Cref{tautological_models} is smooth outside a normal crossing divisor $\Delta\subset Y$. Since the Weierstraß model $p\colon W(\mathcal{L},\alpha,\beta)\rightarrow Y$ is smooth over $Y\setminus\Delta$ by \Cref{existence_unique_eq_Weierstrass_model}, $\mathrm{div}(4\alpha^3+27\beta^2)\subset\Delta$. Therefore, $\mathrm{div}(4\alpha^3+27\beta^2)$ has normal crossings as well. Applying \Cref{Weierstrass_model_canonical_sing_canonical_bundle} and \Cref{canonical_sing_plus_Cartier_ascends_if_quasi-etale} shows that $T$ has canonical singularities, that $\mathcal{K}_T$ is invertible, and that
    \begin{equation*}
        \mathcal{K}_T\cong \mathrm{mult}_m^*\mathcal{K}_{W(\mathcal{L},\alpha,\beta)^{\eta_G'}}\cong \mathrm{mult}_m^*(p^{\eta_G'})^*(\mathcal{K}_Y\otimes\mathcal{L}^\vee)\cong q^*(\mathcal{K}_Y\otimes\mathcal{L}^\vee).
    \end{equation*}
\end{proof}
\section{Families of elliptic fibrations and algebraic approximations}
In this section, let $f\colon X\rightarrow Y$ be a $G$-equivariant elliptic fibration from a compact complex analytic variety $X$ bimeromorphic to a compact K\"ahler manifold to a smooth projective variety $Y$. Suppose that the discriminant divisor of $f$ is a normal crossing divisor and that $f$ admits local meromorphic sections over every point of $Y$. Denote by $H:=R^1f_*\Z\vert_{Y^*}$ the variation of Hodge structures induced by $f$ and denote by $p\colon W(\mathcal{L},\alpha,\beta)\rightarrow Y$ the unique $G$-equivariant minimal Weierstraß model associated to $(Y,\Delta,H)$ by \Cref{existence_unique_eq_Weierstrass_model}. Let $\eta_G\in H^1_G(Y,\mathcal{W}^{\mathrm{mer}})$ be the image of the class of $f$ under the map $\mathcal{E}_G(Y,\Delta,H)\rightarrow H^1_G(Y,\mathcal{W}^{\mathrm{mer}})$ from \cref{injection_into_H1(Wmer)} above. Fix a positive integer $m$ such that $m\eta_G$ can be lifted to a class $\eta'_G\in H^1_G(Y,\mathcal{W}^\#)$, \Cref{tautological_models} \cref{multiple_of_eta_lifts}. Denote the $G$-equivariant tautological model associated to $f$ and $m$ by $q\colon T\rightarrow Y$.
\begin{proposition}\label{Weierstrass_families}\cite[][3.1.4, Def. 2.2]{Lin_algebraic_approximation_codim_1}
    Fix a cohomology class $\theta_G\in H_G^1(Y,\mathcal{W}^\#)$, and denote its image in $H^1(Y,\mathcal{W}^\#)$ by $\theta$. Then there exists a family of elliptic fibrations
    \begin{equation*}
        \Pi'\colon\mathcal{W}\stackrel{P}\longrightarrow Y\times V\longrightarrow V:=H^1(Y,\mathcal{L})
    \end{equation*}
    satisfying the following properties.
    \begin{enumerate}[label=(\roman*)]
        \item\label{Weierstrass_family_locally_trivial} The family is locally trivial over $Y$, i.e. there is an open cover $\{U_i\}_{i\in I}$ of $Y$ such that there is an isomorphism
        \begin{equation*}
            P^{-1}(U_i\times V)\cong P_0^{-1}(U_i)\times V
        \end{equation*}
        for each $i\in I$, where $P_0\colon\mathcal{W}_0:=\Pi'^{-1}(0)\rightarrow Y\times\{0\}$ is the central fiber of $\Pi'$.
        \item Denote by $\exp\colon V\rightarrow H^1(Y,\mathcal{W}^\#)$ the map induced by \cref{ses_vhs_line_bundle_W}. Then for each $v\in V$, the fiber $P_v\colon\mathcal{W}_v:=\Pi'^{-1}(v)\rightarrow Y\times\{v\}$ is isomorphic to the twisted Weierstraß model $W(\mathcal{L},\alpha,\beta)^{\theta+\exp(v)}\rightarrow Y$.
        \item Denote by $V^G\subset V$ the $G$-invariant subspace, and denote by $\exp_G\colon V^G\rightarrow H_G^1(Y,\mathcal{W}^\#)$ the map induced by \cref{ses_vhs_line_bundle_W}. Then there is a $G$-action on $\mathcal{W}^G:=\Pi'^{-1}(V^G)$ such that $\Pi'^G:=\Pi'\vert_{\mathcal{W}^G}$ is $G$-equivariant. For each $v\in V^G$, the fiber $P_v\colon\mathcal{W}_v\rightarrow Y\times\{v\}$ is a $G$-equivariant elliptic fibration which corresponds to the class $\theta_G+\exp_G(v)\in H_G^1(Y,\mathcal{W}^\#)$.
        \item The quotient family
        \begin{equation*}
            \Pi'^G/G\colon\mathcal{W}^G/G\longrightarrow Y/G\times V^G\longrightarrow V^G
        \end{equation*}
        is locally trivial over $Y/G$.
    \end{enumerate}
\end{proposition}
\begin{proposition}\label{Kaehler_when_cocycle_is_mapped_to_torsion_class}\cite[][Prop. 3.23]{Claudon_Hoering_Lin_fundamental_group}
    In the situation and notation of \Cref{Weierstrass_families}, suppose that the short exact sequence \cref{ses_vhs_line_bundle_W} sends $\theta$ to a torsion class in $H^2(Y,j_*H_\Z)$. Then for all $v\in V$ the fiber $\mathcal{W}_v$ is bimeromorphic to a compact K\"ahler manifold.
\end{proposition}
\begin{proposition}\label{tautological_families}\cite[][Prop.-Def. 3.17, Def. 2.2]{Lin_algebraic_approximation_codim_1}
    There exists a family of elliptic fibrations
    \begin{equation*}
        \Pi\colon\mathcal{T}\stackrel{Q}{\longrightarrow}Y\times V\longrightarrow V:=H^1(Y,\mathcal{L})
    \end{equation*}
    satisfying the following properties.
    \begin{enumerate}[label=(\roman*)]
        \item The central fiber $Q_o\colon\mathcal{T}_0\rightarrow Y\times\{0\}$ is isomorphic to the fixed $G$-equivariant tautological model $q\colon T\rightarrow Y$ in \Cref{tautological_models}.
        \item The family is locally trivial over $Y$, i.e. there is an open cover $\{U_i\}_{i\in I}$ of $Y$ such that there is an isomorphism
        \begin{equation*}
            Q^{-1}(U_i\times V)\cong q^{-1}(U_i)\times V
        \end{equation*}
        for each $i\in I$.
        \item Denote by $\exp'\colon V\rightarrow H^1(Y,\mathcal{W}^{\mathrm{mer}})$ the map induced by \cref{ses_vhs_line_bundle_W} composed with $\mathcal{W}^\#\hookrightarrow\mathcal{W}^{\mathrm{mer}}$. Then for every $v\in V$ the fiber $Q_v\colon\mathcal{T}_v:=\Pi^{-1}(v)\rightarrow Y\times\{v\}$ is isomorphic to a tautological model of the class $\eta+\exp'(v)\in H^1(Y,\mathcal{W}^{\mathrm{mer}})$, where $\eta$ is the image of $\eta_G$ via the natural map $H_G^1(Y,\mathcal{W}^{\mathrm{mer}})\rightarrow H^1(Y,\mathcal{W}^{\mathrm{mer}})$.
        \item Denote by $V^G\subset V$ the $G$-invariant subspace, and denote by $\exp'_G\colon V^G\rightarrow H_G^1(Y,\mathcal{W}^{\mathrm{mer}})$ the map induced by \cref{ses_vhs_line_bundle_W} composed with $\mathcal{W}^\#\hookrightarrow\mathcal{W}^{\mathrm{mer}}$. Then there is a $G$-action on $\mathcal{T}^G:=\Pi^{-1}(V^G)$ such that $\Pi^G:=\Pi\vert_{\mathcal{T}^G}$ is $G$-equivariant. For each $v\in V^G$, the fiber $\mathcal{T}_v\rightarrow Y\times\{v\}$ is a $G$-equivariant elliptic fibration which corresponds to the class $\eta_G+\exp'_G(v)\in H_G^1(Y,\mathcal{W}^{\mathrm{mer}})$.
        \item The quotient family
        \begin{equation*}
            \Pi^G/G\colon\mathcal{T}^G/G\longrightarrow Y/G\times V^G\longrightarrow V^G
        \end{equation*}
        is locally trivial over $Y/G$.
        \item\label{diagram_tautological_m_Weierstrass} Let $\Pi'\colon\mathcal{W}\rightarrow Y\times V\rightarrow V$ be the family associated to the image $\eta'\in H^1(Y,\mathcal{W}^\#)$ of $\eta'_G$ from \Cref{Weierstrass_families}. Then there exists a morphism $\mathrm{Mult}_m\colon\mathcal{T}\rightarrow\mathcal{W}$ over $Y\times V$ whose restriction to each fiber over $V$ is the multiplication by $m$-map defined in \Cref{tautological_models}. Over $V^G$, the map $\mathrm{Mult}_m$ is $G$-equivariant.
    \end{enumerate}
\end{proposition}
\begin{definition}
    In the situation and notation of \Cref{tautological_families}, we call $\Pi^G\colon\mathcal{T}^G\rightarrow Y\times V^G\rightarrow V^G$ the $G$-equivariant tautological family associated to $f\colon X\rightarrow Y$, resp. to $\eta_G\in H_G^1(Y,\mathcal{W}^{\mathrm{mer}})$.
\end{definition}
The following lemma is well-known to experts. For the lack of a reference we give a short proof.
\begin{lemma}\label{Fibers_tautological_family_are_bimeromorphic_to_Kaehler_manifold}
    For all $v\in V$ the fiber $\mathcal{T}_v$ is bimeromorphic to a compact K\"ahler manifold.
\end{lemma}
\begin{proof}
    Denote by $\eta'\in H^1(Y,\mathcal{W}^\#)$ the image of $\eta'_G\in H_G^1(Y,\mathcal{W}^\#)$ under the natural map $H_G^1(Y,\mathcal{W}^\#)\rightarrow H^1(Y,\mathcal{W}^\#)$. Since $X$ is bimeromorphic to a compact K\"ahler manifold, the short exact sequence \cref{ses_vhs_line_bundle_W} sends $\eta'$ to a torsion class in $H^2(Y,j_*H_\Z)$ \cite[][Thm. 3.20]{Claudon_Hoering_Lin_fundamental_group}. Consider the morphism
    \begin{equation*}
        (\mathrm{Mult}_m)_v\colon\mathcal{T}_v\rightarrow\mathcal{W}_v
    \end{equation*}
    from \Cref{tautological_families} \cref{diagram_tautological_m_Weierstrass}. By \Cref{Kaehler_when_cocycle_is_mapped_to_torsion_class}, $\mathcal{W}_v$ is bimeromorphic to a compact K\"ahler manifold. Since $(\mathrm{Mult}_m)_v$ is a finite morphism, also $\mathcal{T}_v$ is bimeromorphic to a compact K\"ahler manifold.
\end{proof}
\begin{lemma}\label{properties_fibers_tautological_family}
    For every $v\in V^G$, the following assertions hold true.
    \begin{enumerate}[label=(\roman*)]
        \item The morphism $(\mathrm{Mult}_m)_v\colon\mathcal{T}_v\rightarrow\mathcal{W}_v$ is quasi-\'etale.
        \item The compact complex analytic variety $\mathcal{T}_v$ has canonical singularities.
        \item The canonical sheaf of $\mathcal{T}_v$ is invertible.
        \item\label{tautological_model_canonical_bundle_twisted_locally_compatible} Fix a $G$-equivariant isomorphism
        \begin{equation*}
            \varphi\colon\mathcal{K}_T\stackrel{\sim}{\longrightarrow}q^*(\mathcal{K}_Y\otimes\mathcal{L}^\vee).
        \end{equation*}
        Choose a $G$-invariant open cover $\{U_i\}_{i\in I}$ of $Y$ over which $Q_v\colon \mathcal{T}_v\rightarrow Y$ is isomorphic to $q\colon T\rightarrow Y$. Fix isomorphisms
        \begin{equation*}
            \gamma_i\colon Q_v^{-1}(U_i)\stackrel{\sim}{\longrightarrow} q^{-1}(U_i)
        \end{equation*}
        over $U_i$ for all $i\in I$. Then there is a $G$-equivariant isomorphism
        \begin{equation*}
            \varphi_v\colon\mathcal{K}_{\mathcal{T}_v}\stackrel{\sim}{\longrightarrow}Q_v^*(\mathcal{K}_Y\otimes\mathcal{L}^\vee)
        \end{equation*}
        such that the composition
        \begin{align*}
            \mathcal{K}_T\vert_{q^{-1}(U_i)}&=(\gamma_i^{-1})^*\gamma_i^*\mathcal{K}_T\vert_{q^{-1}(U_i)}\\
            &\stackrel{d\gamma_i}{\longrightarrow}\gamma_i^*\mathcal{K}_{\mathcal{T}_v}\vert_{q^{-1}(U_i)}\\
            &\stackrel{\gamma_i^*\varphi_v}{\longrightarrow}\gamma_i^*Q_v^*(\mathcal{K}_Y\otimes\mathcal{L}^\vee)\vert_{q^{-1}(U_i)}\\
            &=q^*(\mathcal{K}_Y\otimes\mathcal{L}^\vee)\vert_{q^{-1}(U_i)}
        \end{align*}
        equals $\varphi\vert_{q^{-1}(U_i)}$ for all $i\in I$.
    \end{enumerate}
\end{lemma}
\begin{proof}
    By \Cref{tautological_families}, $Q_v\colon\mathcal{T}_v\rightarrow Y$ is isomorphic to a $G$-equivariant tautological model associated of the class $\eta+\exp'(v)$. We can therefore apply \Cref{properties_tautological_models}. This proves the first three statements.\par
    To prove the last assertion, note that since $(\mathrm{Mult_m)_v\colon\mathcal{T}_v\rightarrow\mathcal{W}_v}$ is finite quasi-\'etale, there is a canonical isomorphism
    \begin{equation*}
        \mathcal{K}_T\cong (\mathrm{Mult}_m)_v^*\mathcal{K}_{\mathcal{W}_v}.
    \end{equation*}
    Thus, pulling back the isomorphism from \Cref{Weierstrass_model_canonical_sing_canonical_bundle} \cref{Wmodel_canonical_bundle_twisted_locally_compatible}, yields the desired isomorphism.
\end{proof}
\begin{proposition}\label{Kodaira_dim_constanttautological_family}
    The Kodaira dimension of $\mathcal{T}_v/G$ is constant in $v\in V^G$.
\end{proposition}
\begin{proof}
    For a vector $v\in V^G$ let $\mu_v\colon R_v\rightarrow\mathcal{T}_v/G$ be a resolution of singularities.
    In the following, we will construct $G$-equivariant coherent sheaves $\mathcal{G}(m)$ on $Y$ for all $m\geq 0$, and an isomorphism of graded rings
    \begin{equation*}
        \bigoplus\limits_{m\geq 0}H^0(Y,\mathcal{G}(m))^G\stackrel{\sim}{\longrightarrow}\bigoplus\limits_{m\geq 0}H^0(R_v,\mathcal{K}_{R_v}^m).
    \end{equation*}
    This proves \Cref{Kodaira_dim_constanttautological_family} as the right hand side is independent of $v\in V^G$.\par
    Due to the existence of functorial resolutions of singularities, we can choose a resolution of singularities $\varphi\colon S_v\rightarrow\mathcal{T}_v\times_{\mathcal{T}_v/G}R_v$ such that the $G$-action on $\mathcal{T}_v$ extends to a $G$-action on $S_v$. Consider the following diagram
    \[\begin{tikzcd}
	   {S_v} & {R_v} & \\
	   {\mathcal{T}_v} & {\mathcal{T}_v/G} \\
	   Y & {Y/G.}
	   \arrow["\rho", from=1-1, to=1-2]
	   \arrow["{\nu_v}"', from=1-1, to=2-1]
	   \arrow["{\mu_v}", from=1-2, to=2-2]
	   \arrow["{\xi_v}", from=2-1, to=2-2]
	   \arrow["{Q_v}"', from=2-1, to=3-1]
	   \arrow["{Q_v/G}", from=2-2, to=3-2]
	   \arrow["\pi", from=3-1, to=3-2]
    \end{tikzcd}\]
    Define the coherent sheaf $\mathcal{F}_v(m):=(\mu_v)_*\mathcal{K}_{R_v}^m$ for all $m\geq 0$. Recall that this coherent sheaf is a local invariant of $\mathcal{T}_v/G$. Denote by $p_1\colon\mathcal{T}_v\times_{\mathcal{T}_v/G}R_v\rightarrow\mathcal{T}_v$, resp. $p_2\colon\mathcal{T}_v\times_{\mathcal{T}_v/G}R_v\rightarrow R_v$ the projections. For all $m\geq 0$, consider the following composition of $G$-equivariant maps
    \begin{align*}
        \xi_v^*\mathcal{F}_v(m)=\xi_v^*(\mu_v)_*\mathcal{K}_{R_v}^m&\longrightarrow (p_1)_*p_2^*\mathcal{K}_{R_v}^m\\
        &\longrightarrow (p_1)_*\varphi_*\varphi^*p_2^*\mathcal{K}_{R_v}^m=(\nu_v)_*\rho^*\mathcal{K}_{R_v}^m\\
        &\longrightarrow (\nu_v)_*\mathcal{K}_{S_v}^m.
    \end{align*}
    The map in the first line is the natural transformation from the base change. The map in the second line is induced by the co-unit map $p_2^*\mathcal{K}_{R_v}^m\rightarrow\varphi_*\varphi^*p_2^*\mathcal{K}_{R_v}^m$. The identification in the second line follows from the identities $\nu_v=p_1\circ\varphi$, and $\rho=p_2\circ\varphi$. The map in the third row is induced by the map $\rho^*\mathcal{K}_{R_v}^m\rightarrow\mathcal{K}_{S_v}^m$. As $\mathcal{T}_v$ has canonical singularities, and since $\mathcal{K}_{\mathcal{T}_v}$ is invertible by \Cref{properties_fibers_tautological_family}, the canonical map $\mathcal{K}_{\mathcal{T}_v}^m\rightarrow (\nu_v)_*\mathcal{K}_{S_v}^m$ is a $G$-equivariant isomorphism for all $m\geq 0$. We have thus constructed the $G$-equivariant maps
    \begin{equation*}
        \xi_v^*\mathcal{F}_v(m)\longrightarrow (\nu_v)_*\rho^*\mathcal{K}_{R_v}^m\longhookrightarrow \mathcal{K}_{\mathcal{T}_v}^m
    \end{equation*}
    for all $m\geq 0$. Pushing these maps forward along $Q_v$ results in the $G$-equivariant maps
    \begin{equation*}
        (Q_v)_*\xi_v^*\mathcal{F}_v(m)\longrightarrow (Q_v\circ\nu_v)_*\rho^*\mathcal{K}_{R_v}^m\longhookrightarrow (Q_v)_*\mathcal{K}_{\mathcal{T}_v}^m.
    \end{equation*}
    For all $m\geq 0$, define the coherent sheaf
    \begin{equation*}
        \mathcal{G}_v(m):=\im\left((Q_v)_*\xi_v^*\mathcal{F}_v(m)\longrightarrow (Q_v)_*\mathcal{K}_{\mathcal{T}_v}^m\right).
    \end{equation*}
    Note that the map $(Q_v)_*\xi_v^*\mathcal{F}_v(m)\rightarrow (Q_v\circ\nu_v)_*\rho^*\mathcal{K}_{R_v}^m$ from above factors as
    \begin{equation*}
        (Q_v)_*\xi_v^*\mathcal{F}_v(m)\longrightarrow\mathcal{G}_v(m)\longhookrightarrow (Q_v\circ\nu_v)_*\rho^*\mathcal{K}_{R_v}^m.
    \end{equation*}
    Thus, after taking $G$-invariant global sections, we obtain the commutative diagram
    \[\begin{tikzcd}
	   {H^0(Y,(Q_v)_*\xi_v^*\mathcal{F}_v(m))^G} && {H^0(Y,(Q_v\circ\nu_v)_*\rho^*\mathcal{K}_{R_v}^m)^G} \\
	   & {H^0(Y,\mathcal{G}_v(m))^G}
	   \arrow["{a_m}", from=1-1, to=1-3]
	   \arrow[from=1-1, to=2-2]
	   \arrow["{b_m}"', hook, from=2-2, to=1-3]
    \end{tikzcd}\]
    Note that we can identify the maps $a_m$ with the maps
    \begin{equation*}
        H^0(\mathcal{T}_v,\xi_v^*\mathcal{F}_v(m))^G\longrightarrow H^0(\mathcal{T}_v,(\nu_v)_*\rho^*\mathcal{K}_{R_v}^m)^G
    \end{equation*}
    for all $m\geq 0$. By using the isomorphisms
    \begin{align*}
        H^0(\mathcal{T}_v/G,\mathcal{F}_v(m))&\stackrel{\sim}{\longrightarrow}H^0(\mathcal{T}_v,\xi_v^*\mathcal{F}_v(m))^G,\;\mathrm{and}\\
        H^0(R_v,\mathcal{K}_{R_v}^m)&\stackrel{\sim}{\longrightarrow}H^0(S_v,\rho^*\mathcal{K}_{R_v}^m)^G\stackrel{\sim}{\longrightarrow}H^0(\mathcal{T}_v,(\nu_v)_*\rho^*\mathcal{K}_{R_v}^m),
    \end{align*}
    we can identify the maps $a_m$ further with the canonical isomorphisms
    \begin{equation*}
        H^0(\mathcal{T}_v/G,\mathcal{F}_v(m))\stackrel{\sim}{\longrightarrow}H^0(R_v,\mathcal{K}_{R_v}^m).
    \end{equation*}
    In particular, the maps $a_m$ are isomorphisms for all $m\geq 0$. Thus, also the maps $b_m$ are isomorphisms for all $m\geq 0$.\par
    By composing $b_m$ with the canonical isomorphisms
    \begin{equation*}
        H^0(Y,(Q_v\circ\nu_v)_*\rho^*\mathcal{K}_{R_v}^m)^G\cong H^0(S_v,\rho^*\mathcal{K}_{R_v}^m)^G\cong H^0(R_v,\mathcal{K}_{R_v}^m),
    \end{equation*}
    we obtain isomorphisms
    \begin{equation*}
        c_m\colon H^0(Y,\mathcal{G}_v(m))^G\stackrel{\sim}{\longrightarrow} H^0(R_v,\mathcal{K}_{R_v}^m).
    \end{equation*}
    These isomorphisms give rise to an isomorphisms of graded rings
    \begin{equation*}
        \bigoplus\limits_{m\geq 0}H^0(Y,\mathcal{G}_v(m))^G\stackrel{\sim}{\longrightarrow}\bigoplus\limits_{m\geq 0}H^0(R_v,\mathcal{K}_{R_v}^m).
    \end{equation*}
    To conclude the proof, we show that the sheaves $\mathcal{G}_v(m)$ do not depend on $v\in V^G$. Recall that the elliptic fibration $Q_v/G\colon\mathcal{T}_v/G\rightarrow Y/G$ is locally over $Y/G$ isomorphic to $q=Q_0\colon T/G=\mathcal{T}_0/G\rightarrow Y/G$ by \Cref{tautological_families}. Thus, we can find a $G$-invariant open cover $\{U_i\}_{i\in I}$ of $Y$ such that there are $G$-equivariant isomorphisms
    \begin{equation*}
        \gamma_i\colon Q_v^{-1}(U_i)\stackrel{\sim}{\longrightarrow} Q_0^{-1}(U_i)
    \end{equation*}
    over $U_i$ for all $i\in I$. As the the sheaves $\mathcal{F}_v(m)$ are independent of the choice of the resolution of singularities $\mu_v\colon R_v\rightarrow\mathcal{T}_v/G$, there are isomorphisms
    \begin{equation*}
        \gamma_i^*\mathcal{F}_0(m)\vert_{(Q_0/G)^{-1}(U_i/G)}\stackrel{\sim}{\longrightarrow}\mathcal{F}_v(m)\vert_{(Q_v/G)^{-1}(U_i/G)}
    \end{equation*}
    for all $i\in I$, which only depend on the choice of local isomorphisms $\gamma_i$. Pulling these canonical isomorphisms back along $\xi_v$, resp. $\xi_0$, and pushing them forward along $Q_v$, resp. $Q_0$ results in $G$-equivariant isomorphisms
    \begin{equation*}
        (Q_0)_*\xi_0^*\mathcal{F}_0(m)\vert_{U_i}\stackrel{\sim}{\longrightarrow}(Q_v)_*\xi_v^*\mathcal{F}_v(m)\vert_{U_i}.
    \end{equation*}
    Note that the $G$-equivariant isomorphisms
    \begin{equation*}
        \gamma_i\colon Q_v^{-1}(U_i)\stackrel{\sim}{\longrightarrow} Q_0^{-1}(U_i)
    \end{equation*}
    induce $G$-equivariant isomorphisms
    \begin{equation*}
        \gamma_i^*\mathcal{K}_{\mathcal{T}_0}^m\vert_{Q_0^{-1}(U_i)}\stackrel{\sim}{\longrightarrow}\mathcal{K}_{\mathcal{T}_v}^m\vert_{Q_v^{-1}(U_i)}
    \end{equation*}
    for all $i\in I$. As the maps
    \begin{align*}
        \xi_v^*\mathcal{F}_v(m)&\longrightarrow\mathcal{K}_{\mathcal{T}_v}^m,\;\mathrm{resp.}\\
        \xi_0^*\mathcal{F}_0(m)&\longrightarrow\mathcal{K}_{\mathcal{T}_0}^m
    \end{align*}
    were constructed without making any choice, we obtain the $G$-equivariant commutative diagram
    \[\begin{tikzcd}
	   {(Q_0)_*\xi_0^*\mathcal{F}_0(m)\vert_{U_i}} & {(Q_v)_*\xi_v^*\mathcal{F}_v(m)\vert_{U_i}} \\
	   {(Q_0)_*\mathcal{K}_{\mathcal{T}_0}^m\vert_{U_i}} & {(Q_v)_*\mathcal{K}_{\mathcal{T}_v}^m\vert_{U_i}.}
	   \arrow["\sim", from=1-1, to=1-2]
	   \arrow[from=1-1, to=2-1]
	   \arrow[from=1-2, to=2-2]
	   \arrow["\sim", from=2-1, to=2-2]
    \end{tikzcd}\]
    Fix a $G$-equivariant isomorphism
    \begin{equation*}
        \varphi\colon\mathcal{K}_T\stackrel{\sim}{\longrightarrow}q^*(\mathcal{K}_Y\otimes\mathcal{L}^\vee),
    \end{equation*}
    and fix $G$-equivariant isomorphism
    \begin{equation*}
        \varphi_v\colon\mathcal{K}_{\mathcal{T}_v}\stackrel{\sim}{\longrightarrow}Q_v^*(\mathcal{K}_Y\otimes\mathcal{L}^\vee)
    \end{equation*}
    which satisfy the property in \Cref{properties_fibers_tautological_family} \cref{tautological_model_canonical_bundle_twisted_locally_compatible}. We then obtain the $G$-equivariant commutative diagram
    \[\begin{tikzcd}
	   {(Q_0)_*\xi_0^*\mathcal{F}_0(m)\vert_{U_i}} & {(Q_v)_*\xi_v^*\mathcal{F}_v(m)\vert_{U_i}} \\
	   {(Q_0)_*\mathcal{K}_{\mathcal{T}_0}^m\vert_{U_i}} & {(Q_v)_*\mathcal{K}_{\mathcal{T}_v}^m\vert_{U_i}} \\
	   {(\mathcal{K}_Y\otimes\mathcal{L}^\vee)^m} & {(\mathcal{K}_Y\otimes\mathcal{L}^\vee)^m.}
	   \arrow["\sim", from=1-1, to=1-2]
	   \arrow[from=1-1, to=2-1]
	   \arrow[from=1-2, to=2-2]
	   \arrow["\sim", from=2-1, to=2-2]
	   \arrow["{(Q_0)_*\varphi^{\otimes m}\vert_{U_i}}"', from=2-1, to=3-1]
	   \arrow["{(Q_v)_*\varphi_v^{\otimes m}\vert_{U_i}}", from=2-2, to=3-2]
	   \arrow["{\mathrm{id}}", from=3-1, to=3-2]
    \end{tikzcd}\]
    This shows that $\mathcal{G}_0(m)\stackrel{\sim}{\rightarrow}((Q_0)_*\varphi)(\mathcal{G}_0(m))$ and $\mathcal{G}_v(m)\stackrel{\sim}{\rightarrow}((Q_v)_*\varphi_v)(\mathcal{G}_v(m))$ define the same subsheaf in $(\mathcal{K}_Y\otimes\mathcal{L}^\vee)^m$.\par
    Denote this subsheaf by $\mathcal{G}(m)\subset(\mathcal{K}_Y\otimes\mathcal{L}^\vee)^m$. Then there is an isomorphism of graded rings
    \begin{equation*}
        \bigoplus\limits_{m\geq 0}H^0(Y,\mathcal{G}(m))^G\stackrel{\sim}{\longrightarrow}\bigoplus\limits_{m\geq 0}H^0(R_v,\mathcal{K}_{R_v}^m)
    \end{equation*}
    for all $v\in V^G$. This proves \Cref{Kodaira_dim_constanttautological_family}.
\end{proof}
\subsection{Algebraic approximations}
\begin{construction}\label{construction_nice_equivariant_fibrations_and_tautological_families}
    Let $X_0$ be a normal compact complex analytic variety, which is bimeromorphic to a compact K\"ahler manifold. Let $f_0\colon X_0\rightarrow Y_0$ be an elliptic fibration over a smooth projective variety $Y_0$.
    \begin{step}\label{step1_in_construction_nice_equivariant_fibrations_and_tautological_families}
        We construct a $G$-equivariant elliptic fibration $f_2\colon X_2\rightarrow Y_2$, where $G$ is a finite group acting on a compact normal complex analytic K\"ahler variety $X_2$ and a smooth projective variety $Y_2$. The elliptic fibration $f_2\colon X_2\rightarrow Y_2$ has local meromorphic sections over every point in $Y_2$, and is smooth outside a normal crossing divisor $\Delta_2\subset Y_2$. The quotient fibration $f_2/G\colon X_2/G\rightarrow Y_2/G$ is bimeromorphic to $f_0\colon X_0\rightarrow Y_0$. The spaces $X_2/G$ and $Y_2/G$ are both smooth.\par
        Since $X_0$ is bimeromorphic to a compact K\"ahler manifold, there is a compact K\"ahler manifold $\tilde{X_0}$ together with a projective bimeromorphic morphism $\tilde{X_0}\rightarrow X_0$. Note that the composition $\tilde{X_0}\rightarrow X_0\rightarrow Y_0$ is still an elliptic fibration. Indeed, the exceptional locus of $\tilde{X_0}\rightarrow X_0$ is contracted to a subvariety of codimension at least 2 in $X_0$. In particular, the exceptional locus of $\tilde{X_0}\rightarrow X_0$ does not dominate $Y_0$. Therefore, the morphism $\tilde{X_0}\rightarrow X_0\rightarrow Y_0$ has connected fibers the general fiber is an elliptic curve. Thus, $\tilde{X_0}\rightarrow X_0\rightarrow Y_0$ is an elliptic fibration.\par
        Denote by $\Delta_0\subset Y_0$ the discriminant divisor of $\tilde{X_0}\rightarrow Y_0$. Choose a log-resolution $(Y_1,\Delta_1)\rightarrow (Y_0,\Delta_0)$, which is a projective morphism. Choose a resolution of singularities $X_1\rightarrow \tilde{X_0}\times_{Y_0}Y_1$ which is a projective morphism and an isomorphism over the smooth locus of $\tilde{X_0}\times_{Y_0}Y_1\rightarrow Y_1$. Then the induced morphism
        \begin{equation*}
            f_1\colon X_1\longrightarrow\tilde{X_0}\times_{Y_0}Y_1\longrightarrow Y_1
        \end{equation*}
        is an elliptic fibration, which is smooth outside the normal crossing divisor $\Delta_1\subset Y_1$. Since the composition $X_1\rightarrow\tilde{X_0}\times_{Y_0}Y_1\rightarrow\tilde{X_0}$ is a projective morphism to a compact K\"ahler manifold, the complex manifold $X_1$ is also K\"ahler.\par
        Thus, by \cite[][Thm. 3.3.3]{Nakayama_elliptic-fibrations}, $f_1$ is locally projective. Therefore, we can apply \cite[][Prop. 3.11]{Claudon_Hoering_Lin_fundamental_group} to conclude that there exists a finite Galois cover $Y_2\rightarrow Y_1$ by some smooth projective variety $Y_2$ with Galois group $G$ such that the elliptic fibration
        \begin{equation*}
            f_2\colon X_2\longrightarrow Y_2
        \end{equation*}
        from the normalization $X_2$ of $X_1\times_{Y_1}Y_2$ to $Y_2$ is $G$-equivariant, has local meromorphic sections over every point in $Y_2$ and is smooth outside a normal crossing divisor $\Delta_2\subset Y_2$.
    \end{step}
    \begin{step}\label{step2_in_construction_nice_equivariant_fibrations_and_tautological_families}
        Let $Y_2^*:=Y_2\setminus\Delta_2$ be the smooth locus of $f_2$. Denote by $H:=R^1(f_2)_*\Z\vert_{Y_2^*}$ the variation of Hodge structures induced by $f_2$. By \Cref{existence_unique_eq_Weierstrass_model} there exists a unique $G$-equivariant minimal Weierstraß model $p\colon W(\mathcal{L},\alpha,\beta)\rightarrow Y_2$ smooth over $Y_2^*$ that induces $H$ the $G$-equivariant variation of Hodge structures over $Y_2^*$. Let $\eta_G\in H^1_G(Y_2,\mathcal{W}^{\mathrm{mer}})$ be the image of the class of $f_2$ under the map $\mathcal{E}_G(Y_2,\Delta_2,H)\rightarrow H^1_G(Y_2,\mathcal{W}^{\mathrm{mer}})$ from \cref{injection_into_H1(Wmer)} above. Then, by \Cref{tautological_models}, there is a positive integer $m$ such that $m\eta_G$ can be lifted to a class $\eta'_G\in H^1_G(Y_2,\mathcal{W}^\#)$, and there is a commutative diagram
        \[\begin{tikzcd}
	       {T} && {W(\mathcal{L},\alpha,\beta)^{\eta'_G}} \\
	       & {Y_2} && {,}
	       \arrow["{\mathrm{mult}_m}", from=1-1, to=1-3]
	       \arrow["q"', from=1-1, to=2-2]
	       \arrow["p^{\eta'_G}", from=1-3, to=2-2]
        \end{tikzcd}\]
        where $q\colon T\rightarrow Y_2$ is a tautological model of $f_2$, and $\mathrm{mult}_m\colon T\rightarrow W(\mathcal{L},\alpha,\beta)^{\eta'_G}$ is finite of degree $m^2$, which restricts on smooth fibers to the multiplication by $m$-map.\par
        Furthermore, by \Cref{tautological_families} the tautological model $q\colon T\rightarrow Y_2$ fits into a tautological family associated to $f_2$
        \begin{equation*}
            \Pi\colon\mathcal{T}\longrightarrow Y_2\times V\longrightarrow V
        \end{equation*}
        as the central fiber, i.e. $T\cong\mathcal{T}_0$, where $V:=H^1(Y_2,\mathcal{L})$. The restriction of this family to the $G$-invariant subspace $V^G\subset V$ is $G$-equivariant by \Cref{tautological_families}. We can therefore take the quotient to obtain the family
        \begin{equation*}
            \Pi^G/G\colon\mathcal{T}^G/G\stackrel{Q^G/G}{\longrightarrow} Y_2/G\times V^G\longrightarrow V^G.
        \end{equation*}
    \end{step}
\end{construction}
\begin{theorem}\label{theorem_approximation}\cite[][Proof of Thm. 1.2]{Lin_algebraic_approximation_codim_1}
    In the notation of \Cref{construction_nice_equivariant_fibrations_and_tautological_families}, there exists a linear subspace $V'\subset V^G$ and a commutative diagram
    \[\begin{tikzcd}
	   {\mathcal{X}} && {(\Pi^G)^{-1}(V')/G} \\
	   & {V'}
	   \arrow["\sim", dashed, from=1-1, to=1-3]
	   \arrow["\pi"', from=1-1, to=2-2]
	   \arrow["{\Pi^G/G}", from=1-3, to=2-2]
    \end{tikzcd}\]
    such that
    \begin{enumerate}[label=(\roman*)]
        \item $\pi\colon\mathcal{X}\rightarrow V'$ is an algebraic approximation of $X$, i.e. $\pi$ is a smooth and proper morphism such that $\pi^{-1}(0)\cong X$ and the subset
        \begin{equation*}
            A:=\{a\in V'\vert\pi^{-1}(a)\;\mathrm{is\;algebraic}\}\subset V'
        \end{equation*}
        is dense.
        \item The meromorphic map $\mathcal{X}\dashrightarrow (\Pi^G)^{-1}(V')/G$ induces a bimeromorphic map $\mathcal{X}_v\dashrightarrow\mathcal{T}_v/G$ for all $v\in V'$.
    \end{enumerate}
\end{theorem}
\section{Proof of the Main Result}
\begin{theorem}\label{precise_main_result}
    Let $X$ be a compact K\"ahler manifold of Kodaira dimension $\kappa(X)=\dim X-1$ that admits a holomorphic 1-form without zeros. Then there is a smooth projective variety $Y$ of general type and of dimension $\dim Y=\dim X-1$, an elliptic curve $E=\C/\Lambda$, a cohomology class $\theta\in H^1(Y,\OO_Y/\Lambda)$ and a finite \'etale cover $\tilde{X}\rightarrow X$ such that $\tilde{X}$ is bimeromorphic to $(Y\times E)^\theta$.
\end{theorem}
To prove this we need some preparation. For the rest of this section, we establish the following notation. Let $X$ be a compact K\"ahler manifold of Kodaira dimension $\kappa(X)=\dim X-1$ that admits a holomorphic 1-form without zeros. By \cite[][Thm 6.11]{Ueno_classification_theory}, $X$ is bimeromorphic to the total space of an elliptic fibration $f_0\colon X_0\rightarrow Y_0$ over a smooth projective base. Apply \Cref{step1_in_construction_nice_equivariant_fibrations_and_tautological_families} in \Cref{construction_nice_equivariant_fibrations_and_tautological_families} to obtain a $G$-equivariant elliptic fibration
\begin{equation*}
    f_2\colon X_2\rightarrow Y_2,
\end{equation*}
where $G$ is a finite group acting on a compact normal complex analytic K\"ahler variety $X_2$ and on a smooth projective variety $Y_2$. The elliptic fibration $f_2\colon X_2\rightarrow Y_2$ has local meromorphic sections over every point in $Y_2$, and is smooth outside a normal crossing divisor $\Delta_2\subset Y_2$. The quotient fibration $f_2/G\colon X_2/G\rightarrow Y_2/G$ is bimeromorphic to $f_0\colon X_0\rightarrow Y_0$. The spaces $X_2/G$ and $Y_2/G$ are both smooth.\par
Let $Y_2^*:=Y_2\setminus\Delta_2$ be the smooth locus of $f_2$. Denote by $H:=R^1(f_2)_*\Z\vert_{Y_2^*}$ the $G$-equivariant variation of Hodge structures induced by $f_2$. We apply \Cref{step2_in_construction_nice_equivariant_fibrations_and_tautological_families} in \Cref{construction_nice_equivariant_fibrations_and_tautological_families} to obtain the $G$-equivariant minimal Weierstraß model $p\colon W(\mathcal{L},\alpha,\beta)\rightarrow Y_2$ together with the class $\eta_G\in H^1_G(Y_2,\mathcal{W}^{\mathrm{mer}})$ associated to $f_2$, a positive integer $m$ such that $m\eta_G$ lifts to a class $\eta'_G\in H_G^1(Y_2,\mathcal{W}^\#)$. Furthermore, we obtain a $G$-equivariant tautological model $q\colon T\rightarrow Y_2$ associated to $\eta_G$ and $m$, and the $G$-equivariant tautological family
\begin{equation*}
    \Pi^G\colon\mathcal{T}^G\longrightarrow Y_2\times V^G\longrightarrow V^G
\end{equation*}
that includes the $G$-equivariant tautological model $q\colon T\rightarrow Y_2$ as the central fiber.\par
By \Cref{theorem_approximation}, there exists linear subspace $V'\subset V^G$ and a commutative diagram
\[\begin{tikzcd}
	{\mathcal{X}} && {(\Pi^G)^{-1}(V')/G} \\
	& {V'}
	\arrow["\sim", dashed, no head, from=1-1, to=1-3]
	\arrow["\pi"', from=1-1, to=2-2]
	\arrow["{\Pi^G/G}", from=1-3, to=2-2]
\end{tikzcd}\]
such that the family $\mathcal{X}$ is bimeromorphic over $V'$ to $(\Pi^G)^{-1}(V')/G$, the central fiber $\mathcal{X}_0:=\pi^{-1}(0)$ is isomorphic to $X$, and such that the subset
\begin{equation*}
    A:=\{a\in V'\vert\pi^{-1}(a)\;\mathrm{is\;algebraic}\}\subset V'
\end{equation*}
is dense.
\begin{lemma}\label{lemma_deformation}
    There is an open neighborhood $U\subset V'$ of $0\in V'$ such that for all $u\in U$ the fiber $\mathcal{X}_u$ is a compact K\"ahler manifold of Kodaira dimension $\kappa(\mathcal{X}_u)=\dim\mathcal{X}_u-1$ that admits a holomorphic 1-form without zeros.
\end{lemma}
\begin{proof}
    For each vector $v\in V'$, the Kodaira dimension of $\mathcal{X}_v$ equals the Kodaira dimension of $\mathcal{T}_v/G$, which is constant in $v\in V'$ by \Cref{Kodaira_dim_constanttautological_family}. Since $\mathcal{X}_0$ has Kodaira dimension $\kappa(\mathcal{X}_0)=\dim\mathcal{X}_0-1$, it follows that also $\kappa(\mathcal{X}_v)=\dim\mathcal{X}_v-1$ for all $v\in V'$.\par
    Since the central fiber $X=\mathcal{X}_0$ is a compact K\"ahler manifold by assumption, there is an open neighborhood $U\subset V'$ of $0\in V'$ such that $\mathcal{X}_u$ is a compact K\"ahler manifold for all $u\in U$ \cite[][Thm. 15]{Kodaira_Spencer_3}. By \cite[][Prop. 9.20]{Voisin_HT_1}, the map
    \begin{align*}
        U&\longrightarrow\Z\\
        u&\longmapsto\dim H^0(\mathcal{X}_u,\Omega^1_{\mathcal{X}_u})
    \end{align*}
    is constant. We can hence apply Grauert's theorem on proper base change \cite[][Thm. 10.10]{Voisin_HT_1} to conclude that the natural map
    \begin{equation*}
        \pi_*\Omega^1_{\mathcal{X}\vert_U/U}\otimes\C_0\stackrel{\sim}{\longrightarrow} H^0(\mathcal{X}_0,\Omega^1_{\mathcal{X}_0})
    \end{equation*}
    is an isomorphism. Therefore, after shrinking $U$ if necessary, there is a relative holomorphic 1-form $\tilde{\omega}\in H^0(\mathcal{X}\vert_U,\Omega^1_{\mathcal{X}\vert_U/U})$ such that the composition
    \begin{equation*}
        H^0(\mathcal{X}\vert_U,\Omega^1_{\mathcal{X}\vert_U/U})\longrightarrow \pi_*\Omega^1_{\mathcal{X}\vert_U/U}\otimes\C_0\cong H^0(\mathcal{X}_0,\Omega^1_{\mathcal{X}_0})
    \end{equation*}
    sends $\tilde{\omega}$ to $\omega$. Let $Z\subset\mathcal{X}\vert_U$ be the zero locus of $\tilde{\omega}$. This is a closed subvariety, which is disjoint from $\mathcal{X}_0$ by assumption. As $\pi$ is proper, $\pi(Z)\subset U$ is a closed subset that does not contain $0\in U$. In particular, for all $u\in U\setminus\pi(Z)$, the fiber $\mathcal{X}_u$ admits a holomorphic 1-form without zeros. Therefore, replacing $U$ by $U\setminus\pi(Z)$ yields the desired result.
\end{proof}

\begin{lemma}\label{Y_1_of_general_type}
    The following hold true.
    \begin{enumerate}[label=(\roman*)]
        \item Let $a\in A$, and let $X^{\mathrm{min}}_a$ be a good minimal model of $\mathcal{X}_a$. Then the Iitaka fibration $f_a\colon X_a^{\mathrm{min}}\rightarrow Z$ is birational to the elliptic fibration $\mathcal{T}_a/G\rightarrow Y_1$.
        \item The smooth projective variety $Y_1$ is of general type.
    \end{enumerate}
\end{lemma}
\begin{proof}
    Let $X^{\mathrm{min}}_a$ be a good minimal model of $\mathcal{X}_a$, which exists by \cite[][Thm. 4.4]{Lai_varieties_fibered_by_gmm}, and let $f_a\colon X_a^{\mathrm{min}}\rightarrow Z$ be the Iitaka fibration. Let $\mu_a\colon R_a\rightarrow\mathcal{T}_a/G$ be a resolution of singularities. Recall that the proof of \Cref{Kodaira_dim_constanttautological_family} constructed an isomorphism of graded rings
    \begin{equation*}
        \bigoplus\limits_{m\geq 0}H^0(Y_2,\mathcal{G}(m))^G\stackrel{\sim}{\longrightarrow}\bigoplus\limits_{m\geq 0}H^0(R_a,\mathcal{K}_{R_a}^m),
    \end{equation*}
    where the coherent sheaves $\mathcal{G}(m)$ are independent of $a\in A$. Thus, the bimeromorphic map
    \begin{equation*}
        \mathcal{T}_a/G\longdashrightarrow\mathrm{Proj}\left(\bigoplus\limits_{m\geq 0}H^0(R_a,\mathcal{K}_{R_a}^m)\right)\cong\mathrm{Proj}\left(\bigoplus\limits_{m\geq 0}H^0(Y_2,\mathcal{G}(m))^G\right)
    \end{equation*}
    contracts the fibers of $\mathcal{T}_a/G\rightarrow Y_1$. As the birational map
    \begin{equation*}
        \mathcal{T}_a/G\longdashrightarrow\mathrm{Proj}\left(\bigoplus\limits_{m\geq 0}H^0(R_a,\mathcal{K}_{R_a}^m)\right),
    \end{equation*}
    factors birationally through the Iitaka fibration $f_a\colon X_a^{\mathrm{min}}\rightarrow Z$, the Iitaka fibration $f_a\colon X_a^{\mathrm{min}}\rightarrow Z$ contracts the fibers of $\mathcal{T}_a/G\rightarrow Y_1$. Thus, the Iitaka fibration $f_a\colon X_a^{\mathrm{min}}\rightarrow Z$ is birational to the elliptic fibration $\mathcal{T}_a/G\rightarrow Y_1$. This proves the first assertion.\par
    In particular the smooth projective variety $Y_1$ is birational to the base of the Iitaka fibration $Z$. By \cite[][Thm. 3.4]{Hao_Nowhere_Vanishing_Holomorphic_One-Forms_on_Varieties_of_Kodaira_Codimension_One}, $Z$ is of general type. Thus, $Y_1$ is of general type.
\end{proof}
Recall that $\mathcal{T}_v$ has canonical singularities by \Cref{properties_fibers_tautological_family}. As canonical singularities are rational \cite[][Cor. 5.24]{Kollar_Mori}, and as finite quotients of rational singularities are rational \cite[][Rem. 2.1]{Greb_rational_singularities_quotients}, $\mathcal{T}_v/G$ has rational singularities for all $v\in V'\subset V^G$. For each $v\in V'$ fix a resolution of singularities
\begin{equation*}
    \mu_v\colon R_v\longrightarrow\mathcal{T}_v/G.
\end{equation*}
The proof of \cite[][Lem. 8.1]{Kawamata_minimal_models_Kodaira_dim_fiber_spaces} shows that the meromorphic map
\begin{equation*}
    \mathcal{T}_v/G\stackrel{\mu_v^{-1}}{\longdashrightarrow}R_v\stackrel{\mathrm{alb}_{R_v}}{\longrightarrow}\mathrm{Alb}(R_v)
\end{equation*}
extends to a morphism
\begin{equation*}
    \mathrm{alb}_{\mathcal{T}_v/G}\colon\mathcal{T}_v/G\longrightarrow\mathrm{Alb}(\mathcal{T}_v/G):=\mathrm{Alb}(R_v),
\end{equation*}
which we call the Albanese morphism of $\mathcal{T}_v/G$. Note that $\mathrm{Alb}(\mathcal{T}_v/G)\cong\mathrm{Alb}(\mathcal{X}_v)$ for all $v\in V'$.
\begin{lemma}\label{fiber_not_contracted_by_Albanese}
    For each $u\in U$, the general fiber of $\mathcal{T}_u/G\rightarrow Y_2/G=Y_1$ is not contracted by the Albanese morphism $\mathrm{alb}_{\mathcal{T}_u/G}\colon\mathcal{T}_u/G\longrightarrow\mathrm{Alb}(\mathcal{T}_u/G)$. In particular, the $j$-invariant of the smooth fibers of $\mathcal{T}_u\rightarrow Y_2$ is constant for all $u\in U$.
\end{lemma}
\begin{proof}
    By \Cref{theorem_approximation} and \Cref{lemma_deformation}, there is a vector $a\in U$ such that $\mathcal{X}_a$ is a smooth projective variety of Kodaira $\kappa(\mathcal{X}_a)=\dim\mathcal{X}_a-1$ that admits a holomorphic 1-form without zeros. Let $X^{\mathrm{min}}_a$ be a good minimal model of $\mathcal{X}_a$, which exists by \cite[][Thm. 4.4]{Lai_varieties_fibered_by_gmm}, and let $f_a\colon X_a^{\mathrm{min}}\rightarrow Z$ be the Iitaka fibration, which is birational to $\mathcal{T}_a/G\rightarrow Y_1$ by \Cref{Y_1_of_general_type}. By \cite[][Lem. 3.1]{Hao_Nowhere_Vanishing_Holomorphic_One-Forms_on_Varieties_of_Kodaira_Codimension_One}, the general fiber of the Iitaka fibration $f_a\colon X_a^{\mathrm{min}}\rightarrow Z$ is not contracted by the Albanese morphism. Therefore, the general fiber of $\mathcal{T}_a/G\rightarrow Y_1$ is not contracted by the Albanese morphism.\par

    Consider the relative Albanese morphism
    \[\begin{tikzcd}
	   {(\Pi^G)^{-1}(U)/G} && {\mathrm{Alb}_U(\mathcal{X}_u)} \\
	   & U & {.}
	   \arrow["a", from=1-1, to=1-3]
	   \arrow[from=1-1, to=2-2]
	   \arrow[from=1-3, to=2-2]
    \end{tikzcd}\]
    Note that the function
    \begin{align*}
        Y_1\times U&\longrightarrow\Z\\
        (y,u)&\longmapsto\dim a(Q^{-1}(y,u))
    \end{align*}
    is upper semi-continuous. The above argument shows that the function is bounded from below by 1 on $Y_1\times (A\cap U)$. Since $A\cap U$ is dense in $U$, the function is bounded from below by 1 everywhere. Since every fiber of $\mathcal{T}_u/G\rightarrow Y_1$ is 1-dimensional, the function must be constant equal to 1 on all of $U$.\par
    This implies that a smooth fiber $(Q_u/G)^{-1}(y)$ of $Q_u/G\colon\mathcal{T}_u/G\rightarrow Y_1$ is mapped to a translate of a fixed 1-dimensional subtorus $E\subset\mathrm{Alb}(\mathcal{T}_u/G)$. In particular, the morphism $(Q_u/G)^{-1}(y)\rightarrow E$ is finite \'etale. As $E$ admits only countably many finite \'etale covers, and as $Y_1$ is connected, all smooth fibers $(Q_u/G)^{-1}(y)$ give rise to isomorphic finite \'etale covers $(Q_u/G)^{-1}(y)\rightarrow E$. In particular, the $j$-invariant of the smooth fibers $(Q_u/G)^{-1}(y)$ is constant.
\end{proof}
We will need the following result, cf. \cite[][Proof of Thm. 3.3]{Hao_Nowhere_Vanishing_Holomorphic_One-Forms_on_Varieties_of_Kodaira_Codimension_One}, \cite[][Prop. 5.8]{Hao_Schreieder_3-folds}.
\begin{proposition}\label{monodromy_trivial_after_fiinite_etale_cover}
    There is a finite \'etale cover $Y_3\rightarrow Y_1$ such that the elliptic fibration $T/G\times_{Y_1}Y_3\rightarrow Y_3$ has trivial monodromy.
\end{proposition}
\begin{proof}
    Recall that the monodromy of $T/G\rightarrow Y_1$ is defined as the local system $R^1(q/G)_*\Z\vert_{Y_1^*}$. By \Cref{fiber_not_contracted_by_Albanese}, every fiber of $(q/G)^{-1}(Y_1^*)\rightarrow Y_1^*$ is isomorphic to a fixed elliptic curve $E$. Therefore the monodromy is determined by a representation
    \begin{equation*}
        \rho^*\colon\pi_1(Y_1^*)\longrightarrow\mathrm{Aut}(H^1(E,\Z)).
    \end{equation*}
    \begin{claim}\label{claim_representation_factors_through_fundamental_group_of_compactification}
        The representation $\rho^*$ factors through $\pi_1(Y_1^*)\rightarrow\pi_1(Y_1)$.
    \end{claim}
    \begin{proof}[Proof of \Cref{claim_representation_factors_through_fundamental_group_of_compactification}]
        It suffices to show that the kernel of $\pi_1(Y_1^*)\rightarrow\pi_1(Y_1)$ is contained in $\ker(\rho^*)$. The kernel of $\pi_1(Y_1^*)\rightarrow\pi_1(Y_1)$ is normally generated by loops around the discriminant divisor $\Delta_1$.\par
        Fix a very ample line bundle $\mathcal{A}$ on $Y_1$. After replacing $\mathcal{A}$ by $\mathcal{A}^n$ for a suitable $n>0$ large enough, we can assume that a general complete intersection of global sections of $\mathcal{A}^n$ give rise to a smooth projective curve $C\subset Y_1$ such that the genus of $C$ is positive, and such that $C$ meets the discriminant locus $\Delta_1$ transversely. Since $T/G$ is bimeromorphic to a compact K\"ahler manifold, there is a proper bimeromorphic morphism $\overline{T/G}\rightarrow T/G$, where $\overline{T/G}$ is a compact K\"ahler manifold. Denote the image of the exceptional locus of $\overline{T/G}\rightarrow T/G$ in $Y_1$ by $E\subset Y_1$. Note that $E\subsetneq Y_1$ is a proper subvariety. Thus, we can choose $C$ such that it is also not contained in $E$.\par
        Denote by $Z_0\subset T/G\times_{Y_1}C$ the unique irreducible component that dominates $C$. This uses that $C$ is not contained in $\Delta_1$. Let $Z\rightarrow Z_0$ be a minimal resolution of singularities, and denote the induced elliptic fibration by
        \begin{equation*}
            g\colon Z\longrightarrow C.
        \end{equation*}
        Similarly, let $\overline{Z_0}\subset\overline{T/G}\times_{Y_1}C$ be the unique irreducible component that dominates $C$. This uses that $C$ is not contained in $E$. Let $\overline{Z}\rightarrow\overline{Z_0}$ be a minimal resolution of singularities. Note that $\overline{Z_0}$ is a K\"ahler surface as a subvariety of a K\"ahler manifold. Thus, $\overline{Z}$ is a smooth compact K\"ahler surface. Moreover, the natural morphism $\overline{Z_0}\rightarrow Z_0$ is bimeromorphic and extends to a bimeromorphic morphism $\overline{Z}\rightarrow Z$. This shows that $Z$ is a smooth compact K\"ahler surface as well.\par
        The monodromy around each component of $\Delta_1$ that meets $C$ is given by the monodromy of
        \begin{equation*}
            g\colon Z\longrightarrow C
        \end{equation*}
        around $\Delta_1\cap C$. Since $C$ has positive genus, every possible $(-1)$-curve on $Z$ lies in a fiber. Contracting these $(-1)$-curves preserves the elliptic fibration and does not change the monodromy. We can therefore assume that $Z$ is minimal.\par
        Since the general fiber of $T/G\rightarrow Y_1$ is not contracted by the Albanese morphism by \Cref{fiber_not_contracted_by_Albanese}, the commutative diagram
        \[\begin{tikzcd}
	       {\mathrm{Alb}(Z)} & {\mathrm{Alb}(T/G)} \\
	       Z & {T/G} \\
	       C & {Y_1}
	       \arrow[from=1-1, to=1-2]
	       \arrow[from=2-1, to=1-1]
	       \arrow[from=2-1, to=2-2]
	       \arrow[from=2-1, to=3-1]
	       \arrow[from=2-2, to=1-2]
	       \arrow[from=2-2, to=3-2]
	       \arrow[from=3-1, to=3-2]
        \end{tikzcd}\]
        shows that the general fiber of $g$ is also not contracted by the Albanese morphism $\mathrm{alb}_Z\colon Z\rightarrow\mathrm{Alb}(Z)$. Fix a point $c\in \Delta_1\cap C$. Kodaira's classification of singular elliptic fibers of minimal elliptic surfaces \cite[][p. 201]{Barth_Hulek_Peters_vandeVen_surfaces} shows that the monodromy of $g\colon Z\rightarrow C$ around $c\in \Delta_1\cap C$ is non-trivial if and only if $g^{-1}(c)$ is a union of rational curves, and that the monodromy around $c\in \Delta_1\cap C$ is trivial if and only if $g^{-1}(c)$ is a multiple of a smooth elliptic curve.\par
        Recall that rational curves are contracted by any morphism to a torus. Thus, if $g^{-1}(c)$ is a union of rational curves, then the Albanese morphism $\mathrm{alb}_{Z}\colon Z\rightarrow\mathrm{Alb}(Z)$ contracts $g^{-1}(c)$. As the fiber dimension can only jump, it follows that the general fiber of $g\colon Z\rightarrow C$ is contracted by the Albanese morphism, a contradiction. Thus, the monodromy around $c\in \Delta_1\cap C$ is trivial, and the singular fibers of $g\colon Z\rightarrow C$ are multiples of smooth elliptic curves.\par
        Since $Y_1$ can be covered by such curves $C$, the claim follows.
    \end{proof}
    Let $\overline{\rho}\colon\pi_1(Y_1)\rightarrow\mathrm{GL}_2(\Z)$ be the resulting representation.
    \begin{claim}\label{claim_representation_has_finite_image}
        The image of $\overline{\rho}$ is finite.
    \end{claim}
    \begin{proof}[Proof of \Cref{claim_representation_has_finite_image}]
        Construct $g\colon Z\rightarrow C$ as before. The proof of \Cref{claim_representation_factors_through_fundamental_group_of_compactification} shows that all singular fibers of $g\colon Z\rightarrow C$ are multiples of smooth elliptic curves, and that the monodromy of $g$ is given by a representation
        \begin{equation*}
            \rho\colon\pi_1(C)\longrightarrow\mathrm{Aut}(H^1(E,\Z)).
        \end{equation*}
        The multiple fibers of $g\colon Z\rightarrow C$ define an orbifold structure on $C$. Since $C$ has positive genus, the classification of orbifold structures \cite[][Cor. 2.29]{Cooper_Hodgson_Kerckhoff_orbifolds} imply that there is a finite cover $C'\rightarrow C$ branched over the orbifold points with the correct multiplicities. Define $Z'$ as the normalization of $Z\times_CC'$. Then the induced elliptic fibration $g'\colon Z'\rightarrow C'$ is smooth.\par
        The monodromy of $g'$ is given by a representation
        \begin{equation*}
            \rho'\colon\pi_1(C')\longrightarrow\mathrm{Aut}(H^1(E,\Z)).
        \end{equation*}
        Denote by $\varphi\colon\pi_1(C')\rightarrow\pi_1(C)$ the induced map on fundamental groups, and note that $\im(\varphi)\subset\pi_1(C)$ is a finite index subgroup. By construction, $\rho'$ factors through $\rho$. Thus, $\im(\rho)$ is finite if $\im(\rho')$ is finite. Moreover, as $C$ is a complete intersection of ample hyperplanes, the map $\pi_1(C)\rightarrow\pi_1(Y_1)$ is surjective by the Lefschetz hyperplane theorem \cite[][Thm. 7.4]{Milnor_Morse_theory}. Therefore, $\im(\overline{\rho})$ is finite if $\im(\rho)$ is finite. It is thus sufficient to prove that $\im(\rho')$ is finite, which we will do in the following.\par
        Since $H^1(E,\C)\cong H^1(E,\Z)\otimes\C$, it suffices to show that the local system $R^1g'_*\C$ has finite monodromy, i.e. that the monodromy representation
        \begin{equation*}
            \rho'_\C\colon\pi_1(C')\longrightarrow\mathrm{Aut}(H^1(E,\C))
        \end{equation*}
        has finite image. Let $\tilde{C}\rightarrow C'$ be the universal cover of $C'$. Then $\tilde{Z}:=Z'\times_{C'}\tilde{C}\rightarrow\tilde{C}$ is an isotrivial smooth elliptic fibration with fiber $E=\C/\Lambda$, and with trivial monodromy. By \Cref{corollary_isotrivial_and_trivial_vhs},
        \begin{equation*}
            \tilde{Z}\cong (\tilde{C}\times E)^\theta
        \end{equation*}
        for some $\theta\in H^1(\tilde{C},\OO_{\tilde{C}}/\Lambda)$. Since $\tilde{C}$ is a Stein manifold, $H^1(\tilde{C},\OO_{\tilde{C}})=0=H^2(\tilde{C},\Z)$. Therefore, also $H^1(\tilde{C},\OO_{\tilde{C}}/\Lambda)=0$, which shows
        \begin{equation*}
            \tilde{Z}\cong \tilde{C}\times E.
        \end{equation*}
        Denote the $\pi_1(C')$-action on $\tilde{C}$ by
        \begin{align*}
            \sigma_{\tilde{C}}\colon\pi_1(C')\times \tilde{C}&\longrightarrow \tilde{C},\\
            (\gamma,c)&\longmapsto\sigma_{\tilde{C}}(\gamma)(c),
        \end{align*}
        and denote the induced $\pi_1(C')$-action on $\tilde{C}\times E$ by
        \begin{align*}
            \sigma\colon\pi_1(C')\times (\tilde{C}\times E)&\longrightarrow \tilde{C}\times E,\\
            (\gamma,(c,e))&\longmapsto (\sigma_{\tilde{C}}(\gamma)(c),\sigma_E(\gamma)(c,e)).
        \end{align*}
        Note that for each $\gamma\in\pi_1(C')$, the map
        \begin{align*}
            \tilde{C}&\longrightarrow\mathrm{Aut}(E),\\
            c&\longmapsto\sigma_E(\gamma)(c,-)
        \end{align*}
        is a holomorphically varying family of automorphisms of $E$. Recall that the identity component of $\mathrm{Aut}(E)$ is given by translations, i.e. by $E\subset\mathrm{Aut}(E)$. Thus, the class
        \begin{equation*}
            \tau(\gamma):=[\sigma_E(\gamma)(c,-)]\in\mathrm{Aut}(E)/E
        \end{equation*}
        is well-defined. Moreover, a simple computation shows that
        \begin{align*}
            \tau\colon\pi_1(C')&\longrightarrow\mathrm{Aut}(E)/E,\\
            \gamma&\longmapsto\tau(\gamma)
        \end{align*}
        is a group homomorphism. Therefore, the monodromy representation of $\rho'$ factors through $\mathrm{Aut}(E)/E$ as
        \begin{equation*}
            \rho'_\C\colon\pi_1(C')\longrightarrow\mathrm{Aut}(E)/E\longhookrightarrow\mathrm{Aut}(H^1(E,\C)).
        \end{equation*}
        Since $\mathrm{Aut}(E)/E$ is a finite group, $\im(\rho'_\C)$ must be finite.
    \end{proof}
    To conclude the proof, let $H:=\ker(\overline{\rho})\subset\pi_1(Y_1)$. By \Cref{claim_representation_has_finite_image}, $H$ is a finite index subgroup. The subgroup $H\subset\pi_1(Y_1)$ gives rise to a finite \'etale cover $Y_3\rightarrow Y_1=Y_2/G$, which satisfies the requirements in the statement.
\end{proof}
\begin{lemma}\label{construction_X'}
    There is a finite \'etale cover $X'\rightarrow X$ such that $X'$ is bimeromorphic to $T/G\times_{Y_1}Y_3$.
\end{lemma}
\begin{proof}
    By Hironaka's resolution of indeterminacies, there is a compact K\"ahler manifold $\tilde{X}$ together with a bimeromorphic morphism $\tilde{X}\rightarrow X$ such that the bimeromorphic map $X\dashrightarrow T/G$ extends to a bimeromorphic morphism $\tilde{X}\rightarrow T/G$. The composition with $T/G\rightarrow Y_1$ gives a morphism $\tilde{X}\rightarrow Y_1$. Then the fiber product $\tilde{X}\times_{Y_1}Y_3$ is bimeromorphic to $T/G\times_{Y_1}Y_3$, and the projection $\tilde{X}\times_{Y_1}Y_3\rightarrow\tilde{X}$ is a finite \'etale cover. Since $\tilde{X}$ is bimeromorphic to $X$, the finite \'etale cover $\tilde{X}\times_{Y_1}Y_3\rightarrow\tilde{X}$ induces a finite \'etale cover $X'\rightarrow X$ such that $X'$ is bimeromorphic to $T/G\times_{Y_1}Y_3$.
\end{proof}
Define $Y_4:=Y_3\times_{Y_1}Y_2$. Note that the $G$-action on $Y_2$ induces a $G$-action on $Y_4$ such that $Y_4/G\cong Y_3$. Moreover, the projection $\tau\colon Y_4\rightarrow Y_2$ is finite \'etale. In particular, $Y_4$ is smooth.
\begin{lemma}\label{isomorphism_to_G-equivariant_twisted_Wmodel}
    There is an elliptic curve $E=\C/\Lambda$, a $G$-action on $\OO_{Y_4}/\Lambda$, and a $G$-equivariant cohomology class $\eta_G\in H_G^1(Y_4,\OO_{Y_4}/\Lambda)$, such that the $G$-equivariant elliptic fibration $q_4\colon T/G\times_{Y_1}^{\mathrm{norm}}Y_4\rightarrow Y_4$ is $G$-equivariantly isomorphic over $Y_4$ to $(Y_4\times E)^{\eta_G}\rightarrow Y_4$. In particular, $X'$ from \Cref{construction_X'} is bimeromorphic to $(Y_4\times E)^{\eta_G}/G$.
\end{lemma}
\begin{proof}
    Note that there is a canonical isomorphism
    \begin{equation*}
        T/G\times_{Y_1}Y_4\cong(T/G\times_{Y_1}Y_3)\times_{Y_1}Y_2\cong (T/G\times_{Y_1}Y_3)\times_{Y_3}Y_3\times_{Y_1}Y_2\cong (T/G\times_{Y_1}Y_3)\times_{Y_3}Y_4.
    \end{equation*}
    After normalizing both sides, we obtain the isomorphism
    \begin{equation*}
        T/G\times_{Y_1}^{\mathrm{norm}}Y_4\cong(T/G\times_{Y_1}Y_3)\times_{Y_3}^{\mathrm{norm}}Y_4
    \end{equation*}
    over $Y_4$. By \Cref{fiber_not_contracted_by_Albanese}, the elliptic fibration $T/G\rightarrow Y_1$ is generically isotrivial. Therefore, the elliptic fibration $T/G\times_{Y_1}^{\mathrm{norm}}Y_4\rightarrow Y_4$ is generically isotrivial. By \Cref{monodromy_trivial_after_fiinite_etale_cover}, the elliptic fibration $T/G\times_{Y_1}Y_3\rightarrow Y_3$ has trivial monodromy. The above isomorphism shows that the monodromy of $T/G\times_{Y_1}^{\mathrm{norm}}Y_4\rightarrow Y_4$ is given by the pullback of the monodromy of $T/G\times_{Y_1}Y_3$ along $Y_4\rightarrow Y_3$. Thus, the elliptic fibration $T/G\times_{Y_1}^{\mathrm{norm}}Y_4\rightarrow Y_4$ is generically isotrivial and has trivial monodromy.\par
    Note that there is also a canonical $G$-equivariant isomorphism
    
    \begin{equation*}
        T/G\times_{Y_1}Y_4=T/G\times_{Y_1}(Y_2\times_{Y_1}Y_3)\cong (T/G\times_{Y_1}Y_2)\times_{Y_1}Y_3.
    \end{equation*}
    After normalizing, we obtain
    \begin{equation*}
        T/G\times_{Y_1}^{\mathrm{norm}}Y_4\cong(T/G\times_{Y_1}^{\mathrm{norm}}Y_2)\times_{Y_1}Y_3,
    \end{equation*}
    where we used that $(T/G\times_{Y_1}^{\mathrm{norm}}Y_2)\times_{Y_1}Y_3$ is normal since $Y_3\rightarrow Y_1$ is finite \'etale. The universal property of the fiber product induces a $G$-equivariant morphism
    \begin{equation*}
        T\longrightarrow T/G\times_{Y_1}Y_2.
    \end{equation*}
    Note that it is finite of degree one. As $T$ is normal, this morphism induces a $G$-equivariant isomorphism
    \begin{equation*}
        T\stackrel{\sim}{\longrightarrow} T/G\times_{Y_1}^{\mathrm{norm}}Y_2.
    \end{equation*}
    Hence, we obtain a $G$-equivariant isomorphism
    \begin{equation*}
        T/G\times_{Y_1}^{\mathrm{norm}}Y_4\cong T\times_{Y_1}Y_3\cong T\times_{Y_2}Y_2\times_{Y_1}Y_3\cong T\times_{Y_2}Y_4
    \end{equation*}
    over $Y_4$. Recall that the morphism $Y_4\rightarrow Y_2$ is finite \'etale. In particular, $Y_4\rightarrow Y_2$ is flat. Since tautological models behave well under flat pull-backs \cite[][Rem. 3.13]{Lin_algebraic_approximation_codim_1}, the elliptic fibration $q_4\colon T/G\times_{Y_1}^{\mathrm{norm}}Y_4\rightarrow Y_4$ is isomorphic to a $G$-equivariant tautological model over $Y_4$. Moreover, the underlying Weierstraß model is given by
    \begin{equation*}
        p_4\colon W(\tau^*\mathcal{L},\tau^*\alpha,\tau^*\beta)\longrightarrow Y_4.
    \end{equation*}
    Let $\Delta_4\subset Y_4$ be the discriminant divisor of $q_4$, and set $Y_4^*:= Y_4\setminus\Delta_4$. Write the general fiber of $q_4\colon T/G\times_{Y_1}Y_4\rightarrow Y_4$ as
    \begin{equation*}
        E=\{y^2z-(x^3+axz^2+bz^3)=0\}\subset\mathbb{P}^2.
    \end{equation*}
    Since $p_4\colon W(\tau^*\mathcal{L},\tau^*\alpha,\tau^*\beta)\rightarrow Y_4$ is isotrivial with trivial monodromy, and since it has a canonical section, there is by \Cref{corollary_isotrivial_and_trivial_vhs} an isomorphism
    \begin{equation*}
        p_4^{-1}(Y_4^*)\cong Y_4^*\times E.
    \end{equation*}
    Note that the latter can be compactified as $Y_4\times E=W(\OO_{Y_4},a,b)$. Therefore, $W(\OO_{Y_4},a,b)$ is a minimal Weierstraß model, which is bimeromorphic to $W(\tau^*\mathcal{L},\tau^*\alpha,\tau^*\beta)$ over $Y_4$. Since minimal Weierstraß models are unique, there is an isomorphism
    \begin{equation*}
        W(\tau^*\mathcal{L},\tau^*\alpha,\tau^*\beta)\cong W(\OO_{Y_4},a,b)
    \end{equation*}
    over $Y_4$. Note that since $\tau^*\mathcal{L}$ is given by $R^1(p_4)_*\OO_{W(\tau^*\mathcal{L},\tau^*\alpha,\tau^*\beta)}$, we obtain an isomorphism $\tau^*\mathcal{L}\cong\OO_{Y_4}$ that identifies $\tau^*\alpha$ with $a$ and $\tau^*\beta$ with $b$.\par
    Now, write $E=\C/\Lambda$. By \cite[][Lem. 1.3.5]{Nakayama_elliptic-fibrations}, every meromorphic section of $Y_4\times E\rightarrow Y_4$ is holomorphic. Thus, the sheaves $\mathcal{W}(\tau^*\mathcal{L},\tau^*\alpha,\tau^*\beta)^\#$ and $\mathcal{W}(\tau^*\mathcal{L},\tau^*\alpha,\tau^*\beta)^{\mathrm{mer}}$ agree, and, by \Cref{corollary_isotrivial_and_trivial_vhs}, they are isomorphic to $\OO_{Y_4}/\Lambda$. In particular, the composition in \cref{composition_construction_injection_into_H1(Wmer)} is of the form
    \begin{equation*}
        \id\colon H^1_G(Y_4,\OO_{Y_4}/\Lambda)\longrightarrow\mathcal{E}_G(Y_4,\Delta_4,\Lambda\vert_{Y_4^*})\longhookrightarrow H^1_G(Y_4,\OO_{Y_4}/\Lambda).
    \end{equation*}
    Therefore, the class of $q_4\colon T/G\times_{Y_1}^{\mathrm{norm}}Y_4\rightarrow Y_4$ is represented by a cohomology class $\eta_G\in H_G^1(Y_4,\OO_{Y_4}/\Lambda)$. Thus, the elliptic fibration $q_4\colon T/G\times_{Y_1}^{\mathrm{norm}}Y_4\rightarrow Y_4$ is $G$-equivariantly bimeromorphic to
    \begin{equation*}
        (Y_4\times E)^{\eta_G}\longrightarrow Y_4.
    \end{equation*}
    By \Cref{tautological_models_are_twisted_Wmodels_in_good_situation}, the $G$-equivariant tautological model $q_4\colon T/G\times_{Y_1}Y_4\rightarrow Y_4$ is a $G$-equivariant twisted Weierstraß model. Therefore, there is an isomorphism $T/G\times_{Y_1}^{\mathrm{norm}}Y_4\cong (Y_4\times E)^{\eta_G}$ over $Y_4$.
\end{proof}
Denote the image of $\eta_G$ under the canonical map $H_G^1(Y_4,\OO_{Y_4}/\Lambda)\rightarrow H^1(Y_4,\OO_{Y_4}/\Lambda)$ by $\eta\in H^1(Y_4,\OO_{Y_4}/\Lambda)$. By \Cref{facts_Weierstrass_modles}, the $G$-action on $(Y_4\times E)^\eta$ from \Cref{isomorphism_to_G-equivariant_twisted_Wmodel} induces a $G^{op}$-action on $\OO_{Y_4}$. The data of the $G^{op}$-action on $\OO_{Y_4}$ is equivalent to a group homomorphism $\rho\colon G^{op}\rightarrow\C^*$. Denote by $G_1$ the kernel of $\rho$.
\begin{lemma}\label{construction_X''}
    The following assertions hold true.
    \begin{enumerate}[label=(\roman*)]
        \item The line bundle $\OO_{Y_4}$ together with the $G^{op}$-action from \Cref{facts_Weierstrass_modles} descends to a line bundle $\mathcal{M}$ on $Y_4/G\cong Y_3$.
        \item The morphism $Y_4/G_1\rightarrow Y_4/G\cong Y_3$ is finite \'etale. In particular, $Y_4/G_1$ is a smooth projective variety.
        \item There is a finite \'etale cover $X''\rightarrow X'$ such that $X''$ is bimeromorphic to $(Y_4\times E)^\eta/G_1$
    \end{enumerate}
\end{lemma}
\begin{proof}
    To prove the first assertion, by Kempf's theorem \cite[][Thm. 2.3]{Drezet_Narasimhan_Picard_group} it suffices to show that for each $y\in Y_4$ the stabilizer $G(y)\subset G$ of $y\in Y_4$ acts trivially on the stalk $(R^1(\pi^\eta)_*\OO_{(Y_4\times E)^\eta})_y$. Fix $y\in Y_4$, and choose a small $G$-invariant open neighborhood $U\subset Y_4$ of $y\in Y_4$ such that
    \begin{equation*}
        (\pi^\eta)^{-1}(U)\cong U\times E.
    \end{equation*}
    Thus, we can write the stalk $(R^1(\pi^\eta)_*\OO_{(Y_4\times E)^\eta})_y$ as
    \begin{equation*}
        (R^1(\pi^\eta)_*\OO_{(Y_4\times E)^\eta})_y\cong\OO_{Y_4,y}\otimes H^1((\pi^\eta)^{-1}(y),\OO_{(\pi^\eta)^{-1}(y)})\cong\OO_{Y_4,y}\otimes H^1(E,\OO_E).
    \end{equation*}
    The $G(y)$-action on $(R^1(\pi^\eta)_*\OO_{(Y_4\times E)^\eta})_y$ is thus determined by the $G(y)$-action on $(\pi^\eta)^{-1}(y)\cong E$.\par
    By \Cref{fiber_not_contracted_by_Albanese}, the general fiber of the elliptic fibration $T/G\rightarrow Y_2/G=Y_1$ is not contracted by the Albanese morphism $\mathrm{alb}_{T/G}\colon T/G\rightarrow\mathrm{Alb}(T/G)$. Hence, also the general fiber of the elliptic fibration $T/G\times_{Y_1}Y_3\rightarrow Y_3$ is not contracted by the Albanese morphism $\mathrm{alb}_{T/G\times_{Y_1}Y_3}\colon T/G\rightarrow\mathrm{Alb}(T/G\times_{Y_1}Y_3)$. Note that $(T/G\times_{Y_1}^{\mathrm{norm}}Y_4)/G\rightarrow Y_4/G\cong Y_3$ is isomorphic to $T/G\times_{Y_1}Y_3\rightarrow Y_3$. By \Cref{isomorphism_to_G-equivariant_twisted_Wmodel}, the elliptic fibration $T/G\times_{Y_1}^{\mathrm{norm}}Y_4\rightarrow Y_4$ is $G$-equivariantly isomorphic to $(Y_4\times E)^\eta\rightarrow Y_4$ over $Y_4$. Thus, the general fiber of $(Y_4\times E)^\eta/G\rightarrow Y_4/G\cong Y_3$ is not contracted by the Albanese morphism $\mathrm{alb}_{(Y_4\times E)^\eta/G}\colon (Y_4\times E)^\eta/G\rightarrow\mathrm{Alb}((Y_4\times E)^\eta/G)$.\par
    Recall that rational curves are contracted by any morphism to a torus. Note that $E/G(y)$ is either a rational curve or $G(y)$ acts on $E$ by translations. Thus, if $E/G(y)$ is a rational curve, then the Albanese morphism $\mathrm{alb}_{(Y_4\times E)^\eta/G}\colon (Y_4\times E)^\eta/G\rightarrow\mathrm{Alb}((Y_4\times E)^\eta/G)$ contracts $(\pi^\eta)^{-1}(y)$. As the fiber dimension can only jump, it follows that the general fiber of $(Y_4\times E)^\eta/G\rightarrow Y_4/G$ is contracted by the Albanese morphism, a contradiction. Therefore, $G(y)$ acts on $E\cong(\pi^\eta)^{-1}(y)$ by translations. Since translations on $E$ are invisible on $H^1(E,\OO_E)$, the action of $G(y)$ on $(R^1(\pi^\eta)_*\OO_{(Y_4\times E)^\eta})_y$ is trivial.\par
    To prove the second assertion, denote by $\varphi\colon Y_4\rightarrow Y_3$ the quotient map. The choice of an isomorphism $\varphi^*\mathcal{M}\cong\OO_{Y_4}$ induces by adjunction a map
    \begin{equation*}
        \mathcal{M}\longrightarrow\varphi_*\OO_{Y_4}.
    \end{equation*}
    Note that since $\varphi^*\mathcal{M}\cong\OO_{Y_4}$, the line bundle $\mathcal{M}$ defines a torsion element in $\pic(Y_3)$. Let $n$ be the order of $\mathcal{M}\in\pic(Y_3)$. Note that $n=\vert\im\rho\vert=[G:G_1]$. Consider the $\OO_{Y_3}$-algebra
    \begin{equation*}
        \mathcal{A}:=\bigoplus\limits_{i=0}^{n-1}\mathcal{M}^i.
    \end{equation*}
    The maps $\mathcal{M}\rightarrow\varphi_*\OO_{Y_4}$, resp. $\OO_{Y_3}\rightarrow\varphi_*\OO_{Y_4}$ induce maps of rings
    \begin{equation*}
        \OO_{Y_3}\longrightarrow\mathcal{A}\longrightarrow\varphi_*\OO_{Y_4}.
    \end{equation*}
    By adjunction, these maps of rings induce morphisms
    \begin{equation*}
        Y_4\longrightarrow\mathrm{Spec}(\mathcal{A})\longrightarrow Y_3.
    \end{equation*}
    By definition of $G_1$, it acts trivially on $\OO_{Y_4}\cong\varphi^*\mathcal{M}$. Therefore, the map
    \begin{equation*}
        \mathcal{A}\longrightarrow\varphi_*\OO_{Y_4}
    \end{equation*}
    is $G_1$-equivariant, where the left hand side is endowed with the trivial $G_1$-action, and the right hand side is endowed with the $G_1$-action from $Y_4$. Thus, the morphism $Y_4\rightarrow\mathrm{Spec}(\mathcal{A})$ factors through $Y_4/G_1$.\par
    Note that $\mathrm{Spec}(\mathcal{A})\rightarrow Y_3$ is finite Galois \'etale with Galois group $G/G_1\cong\im\rho$. Thus, $\mathrm{Spec}(\mathcal{A})$ is smooth hence normal, and the morphism $\mathrm{Spec}(\mathcal{A})\rightarrow Y_3$ is finite of degree $n=[G:G_1]$.\par
    Thus, the morphism $Y_4/G_1\rightarrow\mathrm{Spec}(\mathcal{A})$ is finite of degree 1, and since $\mathrm{Spec}(\mathcal{A})$ is normal, the morphism $Y_4/G_1\rightarrow\mathrm{Spec}(\mathcal{A})$ is an isomorphism.\par
    To prove the last assertion, we argue as in \Cref{construction_X'}. By Hironaka's resolution of indeterminacies, there is a compact K\"ahler manifold $\tilde{X'}$ together with a bimeromorphic morphism $\tilde{X'}\rightarrow X'$ such that the bimeromorphic map $X'\dashrightarrow (Y_4\times E)^\eta/G$ extends to a bimeromorphic morphism $\tilde{X'}\rightarrow (Y_4\times E)^\eta/G$. The composition with $(Y_4\times E)^\eta/G\rightarrow Y_4/G=Y_3$ gives a morphism $\tilde{X'}\rightarrow Y_3$. Then the fiber product $\tilde{X'}\times_{Y_3}Y_4/G_1$ is bimeromorphic to $(Y_4\times E)^\eta/G_1$, and the projection $\tilde{X'}\times_{Y_3}Y_4/G_1\rightarrow\tilde{X'}$ is a finite \'etale cover. Since $\tilde{X'}$ is bimeromorphic to $X'$, the finite \'etale cover $\tilde{X'}\times_{Y_3}Y_4/G_1\rightarrow\tilde{X'}$ induces a finite \'etale cover $X''\rightarrow X'$ such that $X''$ is bimeromorphic to $(Y_4\times E)^\eta/G_1$.
\end{proof}
Note that the $G_1$-action on $(Y_4\times E)^\eta\rightarrow Y_4$ acts trivially on the monodromy. Thus, we can apply \Cref{trivial_vhs_constant_j_all_invariant}. For each $g\in G_1$ denote by $(Y_4\times E)^\eta\times_gY_4$ the fiber product
\[\begin{tikzcd}
	{(Y_4\times E)^\eta\times_gY_4} & {(Y_4\times E)^\eta} \\
	{Y_4} & {Y_4.}
	\arrow["\psi_g", from=1-1, to=1-2]
	\arrow[from=1-1, to=2-1]
	\arrow[from=1-2, to=2-2]
	\arrow["g", from=2-1, to=2-2]
\end{tikzcd}\]
Then, by \Cref{trivial_vhs_constant_j_all_invariant}, there are isomorphisms
\begin{equation*}
    \varphi_g\colon (Y_4\times E)^\eta\times_gY_4\stackrel{\sim}{\longrightarrow} (Y_4\times E)^\eta
\end{equation*}
over $Y_4$ such that $\varphi_h\circ(\varphi_g\times_hY_4)=\varphi_{gh}$ for all $g$, $h\in G_1$. The isomorphisms $\varphi_g$ only depend on the cohomology class $\eta\in H^1(Y_4,\OO_{Y_4}/\Lambda)^{G_1}$, and not on the $G_1$-action on $(Y_4\times E)^\eta$. Moreover, by \Cref{trivial_vhs_constant_j_all_invariant}, the composition in the upper row of
\[\begin{tikzcd}
	{(Y_4\times E)^\eta} & {(Y_4\times E)^\eta\times_gY_4} & {(Y_4\times E)^\eta} & {(Y_4\times E)^\eta} \\
	{Y_4} & {Y_4} & {Y_4} & {Y_4,}
	\arrow["{\varphi^{-1}_g}", from=1-1, to=1-2]
	\arrow[from=1-1, to=2-1]
	\arrow["{\psi_g}", from=1-2, to=1-3]
	\arrow[from=1-2, to=2-2]
	\arrow["{g^{-1}}", from=1-3, to=1-4]
	\arrow[from=1-3, to=2-3]
	\arrow[from=1-4, to=2-4]
	\arrow["\id", from=2-1, to=2-2]
	\arrow["g", from=2-2, to=2-3]
	\arrow["{g^{-1}}", from=2-3, to=2-4]
\end{tikzcd}\]
equals translation on each fiber by an element $\sigma(g)\in E$. The induced map $\sigma\colon G_1\rightarrow E$ is a group homomorphism, which determines the $G_1$-action on $(Y\times E)^\eta$. Let $G_2$ be the kernel of $\sigma$.
\begin{lemma}\label{quotient_G_2_locally_trivial_fibration}
    The morphism $(Y_4\times E)^\eta/G_2\rightarrow Y_4/G_2$ is a locally trivial elliptic fiber bundle with trivial monodromy over the possibly singular base $Y_4/G_2$.
\end{lemma}
\begin{proof}
    Let $U\subset Y_4$ be a small $G_2$-invariant open subset such that $(Y_4\times E)^\eta\rightarrow Y_4$ becomes trivial over $U$. Fix an isomorphism
    \begin{equation*}
        (\pi^\eta)^{-1}(U)\stackrel{\sim}{\longrightarrow} U\times E
    \end{equation*}
    over $U$. Using this isomorphism, the composition $\psi_g\circ\varphi_g^{-1}$ can be identified with
    \begin{align*}
        U\times E&\longrightarrow U\times E\\
        (u,e)&\longmapsto (gu,e).
    \end{align*}
    Therefore, the automorphism $g\colon (Y_4\times E)^\eta\rightarrow (Y_4\times E)^\eta$ can be identified with
    \begin{align*}
        U\times E&\longrightarrow U\times E\\
        (u,e)&\longmapsto (gu,e+\sigma(g))=(gu,e),
    \end{align*}
    since $G_2=\ker(\sigma)$. Thus, the quotient $(Y_4\times E)^\eta/G_2\rightarrow Y_4/G_2$ is locally isomorphic to $U/G_2\times E\rightarrow U/G_2$.
\end{proof}
Choose a resolution of singularities $Y\rightarrow Y_4/G_2$ such that the $G_1/G_2$-action on $Y_4/G_2$ lifts to a $G_1/G_2$-action on $Y$.
\begin{lemma}
    There is a cohomology class $\theta\in H^1(Y,\OO_Y/\Lambda)$ such that $(Y_4\times E)^\eta/G_2\times_{Y_4/G_2}Y$ is isomorphic to $(Y\times E)^\theta$ over $Y$. Moreover, the $G_1/G_2$-action on $(Y_4\times E)^\eta/G_2\rightarrow Y_4/G_2$ induces a $G_1/G_2$-action on $(Y\times E)^\theta\rightarrow Y$.
\end{lemma}
\begin{proof}
    Note that since $(Y_4\times E)^\eta/G_2\rightarrow Y_4/G_2$ is a locally trivial elliptic fibration by \Cref{quotient_G_2_locally_trivial_fibration}, the same holds true for $(Y_4\times E)^\eta/G_2\times_{Y_4/G_2}Y\rightarrow Y$. Since $G_2$ acts trivially on the monodromy of $(Y_4\times E)^\eta\rightarrow Y_4$, the monodromy of $(Y_4\times E)^\eta/G_2\times_{Y_4/G_2}Y\rightarrow Y$ is trivial as well. Therefore, by \Cref{corollary_isotrivial_and_trivial_vhs}, $(Y_4\times E)^\eta/G_2\times_{Y_4/G_2}Y$ must be isomorphic to $(Y\times E)^\theta$ over $Y$ for some cohomology class $\theta\in H^1(Y,\OO_Y/\Lambda)$.\par
    As $G_1/G_2$ acts compatibly on $(Y_4\times E)^\eta/G_2$, $Y_4/G_2$, and on $Y$, we obtain a $G_1/G_2$-action on $(Y\times E)^\theta$.
\end{proof}
\begin{lemma}
    The group $G_1/G_2$ acts fixed point freely on $(Y\times E)^\theta$.
\end{lemma}
\begin{proof}
    Note that $G_1/G_2$ acts trivially on the monodromy of $(Y\times E)^\theta\rightarrow Y$, by the definition of $G_1$. We can therefore apply \Cref{trivial_vhs_constant_j_all_invariant}. For each $\overline{g}\in G_1/G_2$ denote by $(Y\times E)^\theta\times_{\overline{g}}Y$ the fiber product
    \[\begin{tikzcd}
	   {(Y\times E)^\theta\times_{\overline{g}}Y} & {(Y\times E)^\theta} \\
	   {Y} & {Y.}
	   \arrow["\psi_{\overline{g}}", from=1-1, to=1-2]
	   \arrow[from=1-1, to=2-1]
	   \arrow[from=1-2, to=2-2]
	   \arrow["{\overline{g}}", from=2-1, to=2-2]
    \end{tikzcd}\]
    Then, by \Cref{trivial_vhs_constant_j_all_invariant}, there are isomorphisms
    \begin{equation*}
        \varphi_{\overline{g}}\colon (Y\times E)^\theta\times_{\overline{g}}Y\stackrel{\sim}{\longrightarrow} (Y\times E)^\theta
    \end{equation*}
    over $Y$ such that $\varphi_{\overline{h}}\circ(\varphi_{\overline{g}}\times_{\overline{h}}Y)=\varphi_{\overline{gh}}$ for all $\overline{g}$, $\overline{h}\in G_1/G_2$. The isomorphisms $\varphi_{\overline{g}}$ only depend on the cohomology class $\theta\in H^1(Y,\OO_{Y}/\Lambda)^{G_1/G_2}$, and not on the $G_1/G_2$-action on $(Y\times E)^\theta$. Moreover, by \Cref{trivial_vhs_constant_j_all_invariant}, and by construction of $(Y\times E)^\theta\rightarrow Y$, the composition in the upper row of
    \[\begin{tikzcd}
	   {(Y\times E)^\theta} & {(Y\times E)^\theta\times_{\overline{g}}Y} & {(Y\times E)^\theta} & {(Y\times E)^\theta} \\
	   {Y} & {Y} & {Y} & {Y,}
	   \arrow["{\varphi^{-1}_{\overline{g}}}", from=1-1, to=1-2]
	   \arrow[from=1-1, to=2-1]
	   \arrow["{\psi_{\overline{g}}}", from=1-2, to=1-3]
	   \arrow[from=1-2, to=2-2]
	   \arrow["{\overline{g}^{-1}}", from=1-3, to=1-4]
	   \arrow[from=1-3, to=2-3]
	   \arrow[from=1-4, to=2-4]
	   \arrow["\id", from=2-1, to=2-2]
	   \arrow["{\overline{g}}", from=2-2, to=2-3]
	   \arrow["{\overline{g}^{-1}}", from=2-3, to=2-4]
    \end{tikzcd}\]
    equals translation on each fiber by the element $\sigma(g)\in E$, where $g\in G_1$ is a lift of $\overline{g}\in G_1/G_2$. As $G_2:=\ker(\sigma)$, the induced group homomorphism
    \begin{equation*}
        \overline{\sigma}\colon G_1/G_2\longrightarrow E
    \end{equation*}
    is injective. Thus, the above description of the $G_1/G_2$-action on $(Y\times E)^\theta$ is fixed point free.
\end{proof}
\begin{proof}[Proof of \Cref{precise_main_result}]
    By construction, we have bimeromorphic maps
    \begin{equation*}
        (Y\times E)^\theta/(G_1/G_2)\stackrel{\sim}{\longdashrightarrow} ((Y_4\times E)^\eta/G_2)/(G_1/G_2)\cong (Y_4\times E)^\eta/G_1\stackrel{\sim}{\longdashrightarrow} X'',
    \end{equation*}
    where the last bimeromorphic map comes from \Cref{construction_X''}. Since $G_1/G_2$ acts fixed point freely on the compact K\"ahler manifold $(Y\times E)^\theta$, the morphism
    \begin{equation*}
        (Y\times E)^\theta\longrightarrow (Y\times E)^\theta/(G_1/G_2)
    \end{equation*}
    is finite \'etale.\par
    By Hironaka's resolution of indeterminacies, there is a compact K\"ahler manifold $\tilde{X''}$ together with a bimeromorphic morphism $\tilde{X''}\rightarrow X''$ such that the bimeromorphic map $X''\dashrightarrow (Y\times E)^\theta/(G_1/G_2)$ extends to a bimeromorphic morphism $\tilde{X''}\rightarrow (Y\times E)^\theta/(G_1/G_2)$. Then the projection
    \begin{equation*}
        \tilde{X''}\times_{(Y\times E)^\theta/(G_1/G_2)}(Y\times E)^\theta\longrightarrow\tilde{X''}
    \end{equation*}
    is a finite \'etale cover. Since $\tilde{X''}$ is bimeromorphic to $X''$, there is a finite \'etale cover $X'''\rightarrow X''$ such that $X'''$ is bimeromorphic to $(Y\times E)^\theta$. As $X'''\rightarrow X''\rightarrow X'\rightarrow X$ is a composition of finite \'etale covers, the morphism $X'''\rightarrow X$ is the desired finite \'etale cover.\par
    Note that $\mathcal{K}_{(Y\times E)^\theta}= (\pi^\theta)^*\mathcal{K}_Y$. Thus, as $\kappa(Y)=\kappa(Y\times E)^\theta=\dim Y$, the smooth projective variety $Y$ is of general type. This concludes the proof of \Cref{precise_main_result}.
\end{proof}
\printbibliography
\end{document}